\documentclass[11pt]{amsart}

\usepackage[margin=3cm]{geometry}

\usepackage{amssymb,amsthm,amsfonts,amsmath,url}
\usepackage{graphicx}

\newtheoremstyle{example-style}{5pt}{0pt}{}{}{\scshape}{:}{.5em}{}

\newtheorem{Thm}{Theorem}[section]
\newtheorem{conj}[Thm]{Conjecture}
\newtheorem{lem}[Thm]{Lemma}
\newtheorem{cor}[Thm]{Corollary}
\newtheorem{prop}[Thm]{Proposition}
\newtheorem{Problem}[Thm]{Problem}

\newtheorem{compfact}[Thm]{Computational fact}

\theoremstyle{example-style}
\newtheorem{Example}[Thm]{Example}

\theoremstyle{definition}

\newtheorem{Remark}[Thm]{Remark}

\newcommand{\stirlingtwo}[2]{\genfrac{\{}{\}}{0pt}{}{#1}{#2}}

\numberwithin{equation}{section}

\usepackage[utf8]{inputenc}    
\usepackage[T1]{fontenc}       

\usepackage{enumitem}
\setlist[enumerate]{leftmargin=*, font=\upshape, label=\alph*)} 
\setlist[itemize]{leftmargin=*} 

\usepackage{comment}
\usepackage{float}
\usepackage{url}
\usepackage{xcolor}
\usepackage{graphicx}

\makeatletter
\let\@@pmod\pmod
\DeclareRobustCommand{\pmod}{\@ifstar\@pmods\@@pmod}
\def\@pmods#1{\mkern4mu({\operator@font mod}\mkern 6mu#1)}
\makeatother

\begin{document}
\date{May 1, 2018}

\title[Coefficients and higher order derivatives of cyclotomic polynomials]{Coefficients and higher order derivatives of cyclotomic polynomials: old and new}

\author[Herrera-Poyatos]{Andr\'es Herrera-Poyatos}
\address{Departamento de \'Algebra, Universidad de Granada, E-18071 Granada, Espa\~na}
\email{andreshp9@gmail.com}

\author[Moree]{Pieter Moree}
\address{Max-Planck-Institut f\"ur Mathematik, Vivatsgasse 7, D-53111 Bonn, Germany}
\email{moree@mpim-bonn.mpg.de}

{\def\thefootnote{}
}
\maketitle
\setcounter{footnote}{0}

\begin{abstract} The $n^{th}$ cyclotomic polynomial 
$\Phi_n(x)$ is the minimal polynomial of an $n^{th}$ primitive 
root of unity. Its coefficients are the subject of
intensive study
and some formulas 
are known for them. Here we are interested in formulas
which are valid for all natural numbers $n$.
In these a host of famous number theoretical 
objects such as Bernoulli numbers, Stirling numbers of both kinds and Ramanujan 
sums make their appearance, sometimes even at the same time! 

In this paper we present a survey of these 
formulas which until now were scattered in the literature 
and introduce an unified approach to derive 
some of them, leading
also to shorter proofs as a by-product.
In particular, we show that some of the formulas have a more
elegant reinterpretation in terms of Bell polynomials.
This approach amounts to computing the logarithmic derivatives of $\Phi_n$ at certain points.
Furthermore, we show that the logarithmic derivatives  at $\pm 1$ of any Kronecker polynomial (a monic product of cyclotomic 
polynomials and a monomial) satisfy a family of linear equations whose coefficients are Stirling numbers of the second kind.
We apply these equations to show that certain polynomials are not Kronecker. In particular, we infer
that for every $k\ge 4$ there exists a symmetric
numerical semigroup with embedding dimension $k$ and Frobenius number $2k+1$ that is not cyclotomic, thus  establishing a conjecture of Alexandru Ciolan, Pedro Garc\'ia-S\'anchez and the second author. 
In an appendix Pedro Garc\'ia-S\'anchez shows that
for every $k\ge 4$ there exists a symmetric 
non-cyclotomic numerical semigroup having Frobenius
number $2k+1.$
\end{abstract}
\section{Introduction}
Various aspects of cyclotomic polynomials have been extensively studied from different perspectives, in 
particular their coefficients. 
Let us write the $n^{th}$ cyclotomic polynomial\footnote{In Section \ref{prelim} this and other relevant number theoretical objects are defined and described.} as follows
\begin{equation*}
\label{thangoo1}
\Phi_n(x)=\sum_{j=0}^{\varphi(n)}a_n(j)x^j.
\end{equation*}
The coefficients $a_n(j)$ are usually very small. Indeed, 
in the $19^{th}$ century
mathematicians even
thought that they are always $0$ or $\pm1$. The first counterexample to this claim occurs at $n = 105$, namely one has $a_{105}(7) = -2$. 
Issai Schur in a letter to Edmund Landau sketched an argument
showing that cyclotomic coefficients are unbounded.
His argument is easily adapted to show that 
$\{a_n(j):n\ge 1,~j\ge 0\}=\mathbb Z$, that is,
every integer is assumed as value of a cyclotomic coefficient.
For the best result to date in this direction
see Fintzen \cite{Fintzen}. 
\par Currently computations can be extended enormously far beyond
$n = 105,$ cf. Figure 1. These and theoretical considerations using
analytic number theory 
(cf. various papers of Helmut
Maier \cite{Maier}, who amongst others solved a long standing
conjecture of
Erd\H{o}s on cyclotomic coefficients), show clearly
that the complexity of the coefficients is a function of the number
of distinct odd prime factors of $n$, much rather than the size of
$n$. Complex patterns 
arise (see Figure \ref{fig:coeffs}) and a lot of mysteries remain. 

\begin{figure}
  \centering
  \includegraphics[width=0.75\textwidth]{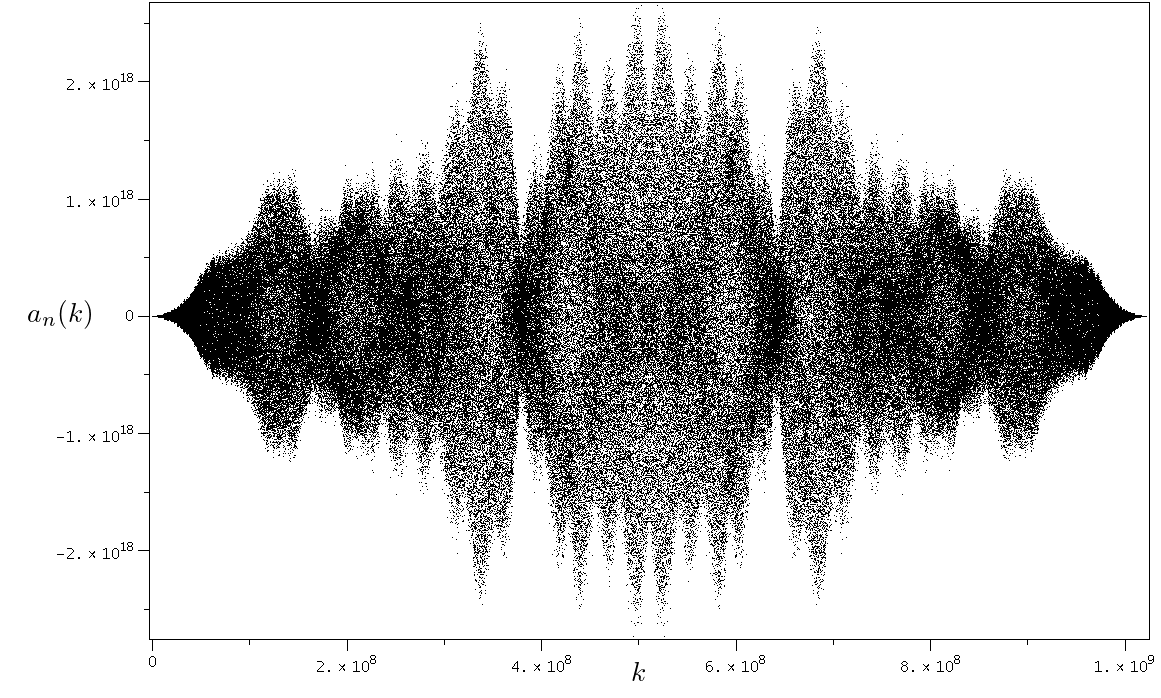}
  \caption{Coefficients of the $n^{th}$ cyclotomic polynomial for $n = 3234846615
 = 3 \cdot 5 \cdot 7 \cdot 11 \cdot 13 \cdot 17 \cdot 19 \cdot 23 \cdot 29$, cf. \cite{AMData}.}
  \label{fig:coeffs}
\end{figure}

This paper has two parts. In the first one, Sections \ref{sec:phi-log-deriv} and \ref{sec:coefficientformulas}, we collect formulas known for 
coefficients of cyclotomic polynomials. 
Some of these formulas are obtained from the derivatives of the function $\log \Phi_n$ at some given point, which we investigate in Section \ref{sec:phi-log-deriv}. 
It seems that Nicol \cite{Nicol} was the first to study
those derivatives in 1962. He showed that
\begin{equation} \label{eq:nicol}
  \log \Phi_n(x) = - \sum_{j = 1}^\infty \frac{r_j(n)}{j} x^j \qquad (|x| < 1, n > 1),
\end{equation}
where the $r_j(n)$ denote Ramanujan sums. 
Four years later D. Lehmer pointed out the connection between the derivatives of $\log \Phi_n$ at $0$ and the coefficients of $\Phi_n$ \cite[Section 2]{Lehmer}.
He also expressed the $k^{th}$ derivative of $\Phi_n$ at $1$ in terms of Bernoulli numbers $B_t^+$, Stirling numbers of the first kind $s(j,t)$ and Jordan totient functions $J_t.$
\begin{Thm}[{\cite[Theorem 3]{Lehmer}}] \label{thm:kth-deriv:1}
  Let $k \ge 1$ and $n\ge 2$ be integers. 
  Then
\[\frac{\Phi_n^{(k)}(1)}{\Phi_n(1)}= k! \sum \prod_{j=1}^k\frac{1}{\lambda_j!} \left(\frac{1}{j!}\sum_{t=1}^{j} \frac{B_t^+ s(j,t)}{t} J_t(n)\right)^{\lambda_j},  \]
  where the sum is over all the non-negative integers $\lambda_1, \ldots, \lambda_k$ such that $\lambda_1 + 2 \lambda_2 + \cdots  + k\lambda_k = k$.
\end{Thm}
As an example let us consider the case $k=2.$ Then one
obtains that for $n \ge 2$ 
\begin{equation}
    \label{formula:k=2}
  \frac{\Phi_n''(1)}{\Phi_n(1)}=\frac{\varphi(n)}{4}\left(\varphi(n) + \frac{\Psi(n)}{3} - 2\right).
\end{equation}  
In this case there are two 
partitions $(\lambda_1,\lambda_2,\ldots)$, namely
$(2,0,\ldots)$ and $(0,1,\ldots)$ giving rise to
a contribution $(\varphi(n)/2)^2,$ respectively 
$-\varphi(n)/2+J_2(n)/12,$ which on adding and using
that $\varphi(n)\Psi(n)=J_2(n)$ gives the result.

We present a new proof of Theorem \ref{thm:kth-deriv:1} 
that consists in determining the values $(\log \Phi_n)^{(k)}(1)$, see Theorem \ref{thm:kth-log-deriv:1}, and employing Fa\`a di Bruno's formula to relate $\Phi_n^{(k)}(1)/\Phi_n(1)$ to the 
$k^{th}$ Bell polynomial, see Theorem \ref{thm:phi:1}. 
We use these findings to derive a novel formula for the $k^{th}$ coefficient of any cyclotomic polynomial (Theorem \ref{thm:phi:coeffs}). The simplest known formula for these coefficients was given by M\"oller in 1971 \cite{HerbertM}. He expressed $a_n(k)$ as a polynomial evaluated in $\mu(n), \mu(n/2), \ldots, \mu(n/k)$, where $\mu$ is the M\"obius function 
with the stipulation that $\mu(t)=0$ if $t$ is not an integer.

\begin{Thm} \label{thm:moller:intro}
  For any non-negative integers $n$ and $k$ with $n \ge 1$ we have
  \begin{equation*}
    a_n(k) = \sum \prod_{j=1}^k (-1)^{\lambda_j}\binom{\mu(n / j)}{\lambda_j},
  \end{equation*}
  where the sum is over all the non-negative integers $\lambda_1, \ldots, \lambda_k$ such that $\lambda_1 + 2 \lambda_2 + \cdots k \lambda_k = k$.
\end{Thm}
A  very short reproof of 
Theorem \ref{thm:moller:intro} was given in \cite{GMH} and is 
recapitulated in 
Section \ref{sec:coefficientformulas}.

During our literature study for this paper we realized that various results claimed as new
at the time, were actually old and in 
Lehmer's and Nicol's papers \cite{Lehmer, Nicol} (although sometimes in round about 
and long form). We urge the reader to have a look at these papers.
One is reminded of a bon mot of Paul Erd\H{o}s: everybody writes and nobody reads.

In the second part of the paper 
comprising Sections \ref{sec:kronecker} and \ref{sec:ns}, we give an application of 
the previously studied results involving 
numerical semigroups (introduced in Section \ref{sec:pre:ns}) and 
Kronecker polynomials. A monic polynomial $f \in \mathbb{Z}[x]$ is said to be a \emph{Kronecker polynomial} if all of its roots are in the closed unit disc. 
Such a polynomial turns out to have a unique factorization in terms
of cyclotomic polynomials and a monomial, a result
due to Kronecker.
In Section \ref{sec:kronecker} we exploit the 
information gathered about the logarithmic derivatives of cyclotomic polynomials at $\pm1$ to obtain a family of identities involving the logarithmic derivatives of Kronecker polynomials at $\pm1$ and Stirling numbers of the second kind $\stirlingtwo{k}{j}$ (see Theorem \ref{thm:log-kronecker:k} for the complete result). 
\begin{Thm} \label{thm:log-kronecker:k:intro}
  Let $f(x) = x^{e_0} \prod_{d \in \mathcal{D}} \Phi_d(x)^{e_d}$ be a Kronecker polynomial. If $f(1) \ne 0$, then, for each integer $k$ with $k \ge 2$ we have
  \[ \sum_{j = 1}^k \stirlingtwo{k}{j} (\log f)^{(j)}(1) = \frac{B_{k}^+}{k} \sum_{d \in \mathcal{D}} e_d J_k(d). \]
\end{Thm}
As $B_{k}^+=0$ for $k\ge 3$ and odd, the logarithmic derivatives of Kronecker polynomials satisfy an infinite number of homogeneous linear equations. These identities can be applied to show that all members of certain infinite families of polynomials are non-Kronecker polynomials and this was our main motivation to initiate the research presented in this paper. In Section \ref{sec:ns} we present an example of such an application. We show that $1 - x + x^k - x^{2k-1} + x^{2k}$ is not Kronecker for every $k\ge 4$ 
(this was proven independently using a different method by Sawhney and Stoner \cite{conjecture-cyclotomic}).
This leads to the following result. 

\begin{Thm} \label{thm:ciolan} ~
  \begin{enumerate}
    \item \label{item:ciolan:e} For every integer $k \ge 4$ there is a symmetric numerical semigroup with embedding dimension $k$ that is not cyclotomic.
    \item \label{item:ciolan:f} For every odd integer $F \ge 9$ there is a symmetric numerical semigroup with Frobenius number $F$ that is not cyclotomic.
  \end{enumerate}
\end{Thm}
Part \ref{item:ciolan:e} establishes the truth of 
a conjecture of Ciolan et al.\,\cite[Conjecture 2]{CGM}.
Pedro A. Garc\'ia-S\'anchez has also found a proof of 
part \ref{item:ciolan:f}. His proof is given in an appendix 
to this paper.

\section{Preliminaries}
\label{prelim}
Here we recall the properties of the relevant number theoretical
notions that are needed for the rest of the paper. References in the subsection headers
give suggestions for further reading.

\subsection{Arithmetic functions \cite{Apostolbook, ScSp, Siva}} \label{sec:arfu}

An \emph{arithmetic function} $f$ maps the positive integers to the complex numbers. 
An important subclass consists of the \emph{multiplicative arithmetic functions};
these have the property that for any two coprime positive integers $a$ and $b$, we have $f(ab)=f(a)f(b)$. 
Famous examples are the Euler totient function, the Dedekind psi function 
and the von Mangoldt function. These are defined, 
respectively, by $\varphi(n)=n\prod_{p \mid n}(1-1/p)$, $\Psi(n)= n \prod_{p \mid n} \left(1 + 1/p\right)$ and
\[ \Lambda(n) = \begin{cases} \log p & \text{if } n = p^k \text{ for some prime } p 
\text{ and integer } k\ge 1; \\ 0 & \text{otherwise.} \end{cases}\]
Another important arithmetic function is the M\"obius function, which is a 
multiplicative and given by
\[\mu(n) = \begin{cases} (-1)^r & \text{if } n \text{ is the product of } r \text{ different primes};  \\ 0 & \text{otherwise}. \end{cases}\]

Let $f$ and $g$ be two arithmetic functions. The \emph{Dirichlet convolution} of $f$ and $g$ is the function $f \star g$ defined as $(f \star g)(n) = \sum_{d \mid n} f(d) g(n / d)$ for every positive integer $n$. The arithmetic functions form a group under the Dirichlet convolution. The neutral element of the group is the function $I(n)$, which equals $1$ if $n$ is $1$ and $0$ otherwise. It turns out that
\begin{equation*} \label{eq:mu}
  \sum_{d \mid n} \mu(d) = I(n).
\end{equation*}
As the Dirichlet convolution of two multiplicative functions is itself multiplicative, the
multiplicative functions are a subgroup of the arithmetic functions.

\subsubsection{The Jordan totient function} \label{gejo}
Jordan totient functions form an important family of multiplicative arithmetic functions.
These functions are a generalization of the Euler totient function \cite{Siva} and they naturally emerge when computing the value $(\log \Phi_n)^{(k)}(1)$.

Let $k\ge 1$ be an integer. 
The $k^{th}$ Jordan totient
function $J_k(n)$ is the number of 
$k$-tuples chosen
from a complete residue system modulo $n$ such that the greatest common divisor
of each set is coprime to $n$. 
It is not difficult to show, 
cf. \cite[p. 91]{Siva}, that
\begin{equation} \label{eq:jordaan}
  n^k=\sum_{d \mid n}J_k(d),
\end{equation}
which by M\"obius inversion yields
\[ J_k(n)=\sum_{d \mid n}\mu\left(\frac{n}{d}\right)d^k. \]
As $J_k$ is a Dirichlet convolution of multiplicative functions, it is itself multiplicative 
and one obtains
\[ J_k(n)=n^k\prod_{p \mid n}\Big(1-\frac{1}{p^k}\Big). \]
Note that $J_1=\varphi$ and that $J_2=J_1\Psi=\varphi \Psi.$ 

Let $r\ge 1$ be an integer and 
${\bar e}=(e_1,\ldots,e_r)$ be a vector with integer
entries. Put $w=\sum_i ie_i$. 
An arithmetic function $J_{\bar e}$ of
the form
$J_{\bar e}= {\prod_{i=1}^{r}J_{i}^{e_i}}$
is said to be a \emph{Jordan totient quotient} of a weight $w$, e.g. 
$J_{(-1,1)}=\Psi$ is of weight one. If 
$w=0$, then we say that $J_{\bar e}$ is a 
\emph{balanced Jordan totient quotient}.
If all integer entries of ${\bar e}$ are non-negative, 
then we say that $J_{\bar e}$ is a \emph{Jordan 
totient product}.

\subsection{Cyclotomic polynomials \cite{Siva, Thanga}}

In this section we recall some relevant properties of cyclotomic polynomials. For proofs see, for instance, 
Thangadurai \cite{Thanga}. We also indicate how to compute the values $\Phi_n'(\pm1)$.

A definition of the $n^{th}$ cyclotomic polynomial is
\begin{equation}
\label{definitie}
\Phi_n(x)=\prod_{1\le j\le n,~(j,n)=1}(x-\zeta_n^j) \in \mathbb C[x],
\end{equation}
where $\zeta_n=e^{2 \pi i / n}$. It is a polynomial of degree $\varphi(n)$ and
actually satisfies $\Phi_n(x)\in \mathbb Z[x]$. Moreover, 
$\Phi_n(x)$ is irreducible over the rationals, cf. 
\cite{Wein}, and
in $\mathbb Q[x]$ we have the factorization into irreducibles
\begin{equation}
\label{facintoirr}
x^n-1=\prod_{d \mid n}\Phi_d(x),
\end{equation}
which by M\"obius inversion yields
$\Phi_n(x)=\prod_{d \, \mid n}(x^d-1)^{\mu(n/d)}$.
Since $\sum_{d \, \mid n}\mu(n/d)=0$ for $n>1$, we can rewrite this as
\begin{equation}
\label{phimoebius}
\Phi_n(x)=\prod_{d \, \mid n} (1-x^d)^{\mu(n/d)}\qquad (n>1).
\end{equation}
Degree comparison in \eqref{facintoirr} leads
to $n=\sum_{d\mid n}\varphi(d)$, which is  identity \eqref{eq:jordaan} with $k=1$.
\begin{lem} \label{basiceqs}
Let $p$ be a prime and $n$ a positive integer. We have
\begin{enumerate}
\item $\Phi_{pn}(x)=\Phi_n(x^p){\rm ~if~}p\mid n$;
\item $\Phi_{pn}(x)=\Phi_n(x^p)/\Phi_n(x){\rm ~if~}p\nmid n$;
\item $\Phi_{2n}(x)= (-1)^{\varphi(n)} \Phi_n(-x){\rm ~if~}2\nmid n$;
\item $\Phi_n(-x)=\Phi_n(x)$ if $4 \mid n$;
\item $\Phi_n(x)=x^{\varphi(n)}\Phi_n(1/x)$, that is, $\Phi_n$ is self-reciprocal if $n > 1$.
\end{enumerate}
\end{lem}
The evaluation of $\Phi_n(1)$ is a classical result.
\begin{lem} 
\label{valueat1A}
We have $\Phi_1(1)=0$ and for $n>1$ we have $\Phi_n(1) = e^{\Lambda(n)}$.
\end{lem}
The value $\Phi_n(-1)$ 
is easily determined on using 
Lemma \ref{basiceqs} once one has calculated 
$\Phi_n(1)$. 
\begin{lem} 
\label{valueat-1}
We have $\Phi_1(-1)=-2$, $\Phi_2(-1)=0$ and for $n>2$ we have
\begin{equation*}
\Phi_n(-1)=\begin{cases}
p & \hbox{if } n=2p^e ;\\
1 & \hbox{otherwise},\end{cases}
\end{equation*}
with $p$ a prime number and $e\ge 1$.
\end{lem}
\noindent For more details regarding the latter two lemmas, the reader
is referred to, e.g., \cite{BHM}.

For self-reciprocal polynomials one can determine the first derivative at $\pm 1$ in most cases. The logarithmic derivative of $\Phi_n$ at $\pm 1$ is obtained as a consequence.

\begin{prop}[{\cite[Lemma 9]{BHM}}] \label{self-reciprocal:deriv}
  Let $f$ be a polynomial of degree $d\ge 1$.\\ 
  Suppose that $f$ is self-reciprocal.
  \begin{enumerate}[label=\alph*)]
    \item We have $f'(1) = f(1) d / 2$;
    \item if $2\nmid d$, then $f(-1)=0$. If $2\mid d$, then $f'(-1)=-f(-1)d/2$.
  \end{enumerate}
\end{prop}

\begin{cor}
\label{phin1:deriv}
  We have $(\log \Phi_n))^{(1)}(\pm 1)=\pm \varphi(n)/2$ for 
  $n > 2$.
\end{cor}

\subsection{Ramanujan sums \cite{Car, Giovanni, ScSp, Siva}}
\label{rama}
Given two positive integers $n$ and $k$,
the Ramanujan sum $r_k(n)$ is defined as the 
(a priori) complex number
\[ r_k(n) = \sum_{\substack{j = 1 \\ (j,n)=1}}^n \zeta_n^{j k}. \]
Trivially $\zeta_n^{k}=\zeta_{n/(n,k)}^{k/(n,k)}$ and the Ramanujan
sum is an element in the ring of integers of $\mathbb{Q}(\zeta_{n/(n,k)})$ 
that is invariant under all Galois automorphisms of that field and thus is
actually an integer.
\par Ramanujan \cite{Ramanujan} used the sums $r_k(n)$ to derive pointwise convergent series representations of
arithmetic functions $g:\mathbb N \to \mathbb C$ of the form
$g(n) =\sum_{k=1}^{\infty}{\hat g}(k)r_k(n)$
with certain coefficients ${\hat g}(k)$. 
A representation of this 
form is called a Ramanujan expansion
of $g$ with Ramanujan coefficients  ${\hat g}(k)$ \cite{Giovanni}.

Some authors before Ramanujan considered 
Ramanujan sums, but only uncovered elementary 
properties, e.g., Kluyver \cite{Kl} 
showed in 1906 that
\begin{equation} \label{eq:holder}
  r_k(n) = \mu\left(\frac{n}{(n,k)}\right) \frac{\varphi(n)}{\varphi(n / (n,k))}
\end{equation}
and 
\begin{equation} \label{eq:kluyver}
  r_k(n) = \sum_{d \mid (n,k)} \mu\left(\frac{n}{d}\right) d.
\end{equation}
The latter formula shows again that $r_k(n)$ is an integer. For proofs of \eqref{eq:holder} 
and \eqref{eq:kluyver} see, e.g., Hardy 
and Wright \cite[Section 16.6]{Hardy}.

\subsection{Bernoulli numbers and polynomials \cite{Arakawa}, \cite[Chapter 9]{Cohen}}
\label{sec:Bernoulli}

Bernoulli numbers can be defined in various ways. 
We will introduce them as 
values of \emph{Bernoulli polynomials}, which are defined by the generating function
\begin{equation} \label{eq:bernoulli-pol:gen}
  \frac{t e^{tx}}{e^t - 1} = \sum_{n = 0}^\infty B_n(x) \frac{t^n}{n!} \qquad (|t| < \pi, x \in \mathbb{R}),
\end{equation}
Note that $B_0(x) = 1$. By multiplying the Taylor series of $t e^{tx} / (e^t - 1)$ and $(e^t-1)/t$ one finds the recursive definition of the Bernoulli polynomials
\begin{equation} \label{eq:bernoulli-pol:re}
B_n(x)= x^n - \sum_{k=0}^{n-1}\binom{n}{k} \frac{B_k(x)}{n-k+1}.
\end{equation}
On comparing  \eqref{eq:bernoulli-pol:gen} with $x$ 
replaced by $1-x$ and \eqref{eq:bernoulli-pol:gen} itself one 
obtains $B_n(1-x) = (-1)^n B_n(x)$. In particular, we have $B_n(0) = (-1)^n B_n(1)$. Moreover, note that
\[\sum_{n = 0}^\infty B_n(1) \frac{t^n}{n!} - \sum_{n = 0}^\infty B_n(1) \frac{(-t)^n}{n!} = \frac{t e^{t}}{e^t - 1} + \frac{t e^{-t}}{e^{-t} - 1} = t\]
and thus we find that $B_n(1) = 0$ for every odd integer $n$ greater than $1$. 

In many sources the $n^{th}$ \emph{Bernoulli number} is defined by evaluating the $n^{th}$ Bernoulli polynomial at $0$, see, for instance, \cite{Apostol}. We must point out that some authors introduce $B_n$ 
instead as the value $B_n(1)$ 
\cite{Arakawa}. These definitions agree for every $n \ne 1$. Moreover, we have $B_1(1) = - B_1(0) = 1/2$. It is convenient to distinguish between both definitions in order to avoid possible inaccuracies. Therefore, we define the sequences $B_n^- = B_n(0)$ and $B_n^+ = B_n(1)$. Indeed, using both definitions at the same time allows us to simplify some equations. For instance, we have 
$\zeta(-n) = (-1)^n B_{n+1}^- / (n+1) 
= - B^+_{n+1} / (n+1)$ for 
every non-negative integer $n$ 
(with $\zeta$ the Riemann zeta
function). Theorem \ref{thm:gessel} below 
provides another example demonstrating the convenience of using both definitions.
We recall that
$\text{sgn}(B_n^{\pm})=(-1)^{ n/2+1}$ for every even positive integer $n$.

From \eqref{eq:bernoulli-pol:re} one sees that the Bernoulli numbers satisfy the following recurrences
\begin{equation*} \label{eq:bernoulli:re}
B_n^-= - \sum_{k=0}^{n-1}\binom{n}{k} \frac{B_k^-}{n-k+1}, \quad
B_n^+= 1 - \sum_{k=0}^{n-1}\binom{n}{k} \frac{B_k^+}{n-k+1},
\end{equation*}
which can be used to compute both sequences.

\subsection{Stirling numbers \cite[Chapter 5]{Comtet}}
\label{sec:pre:stirling}
In our paper we will use the notation and terminology for Stirling numbers
proposed by Knuth \cite{Knuth}.

The \emph{signed Stirling numbers of the first kind} are the coefficients of the polynomial
\begin{equation} \label{eq:stirling1}
x^{\underline{k}} = x(x-1)(x-2) \ldots (x-k+1) = \sum_{j = 0}^k s(k,j) x^j,
\end{equation}
where $k$ is a non-negative integer. The polynomial $x^{\underline{0}}$ is defined as $1$.
By a simple inspection of the definition one obtains the equalities
\begin{align*}
  s(k,k) & = 1 \text{ for every } k \ge 0, \\
  s(k,0) & = 0 \text{ for every } k \ge 1,\\
  s(k,j) & = 0 \text{ for every } j > k \ge 0,\\
  s(k,1) & = (-1)^{k-1}(k-1)! \text{ for every } k \ge 1.
\end{align*}
The sign of $s(k,j)$ is easily determined, yielding $s(k,j) = (-1)^{k-j} |s(k,j)|$. 
It is easy to show that the Stirling numbers of the first kind satisfy the following recurrence 
\begin{equation*} \label{eq:stirling1:re}
s(k,j)=s(k-1,j-1)-(k-1)s(k-1,j) \qquad (k, j \ge 1).
\end{equation*}

Note that $(1+t)^x = \sum_{k = 0}^\infty x^{\underline{k}}\, t^k / k!$ for $|t| < 1$ and, therefore, Stirling numbers of the first kind can also arise 
as coefficients of the generating function
\begin{equation*} \label{eq:stirling1:gf}
  (1 + t)^x = \sum_{k, j = 0}^\infty s(k,j) \frac{t^k}{k!} x^j = \sum_{j = 0}^\infty x^j \sum_{k = j}^\infty s(k,j) \frac{t^k}{k!} \qquad (|t| < 1, x \in \mathbb{R}).
\end{equation*}
Observe that $(1 + t)^x = \sum_{k = 0}^\infty (\log(1+t))^k x^k/k!$ for any $t > -1$ and any real 
number $x.$ It follows that
for each integer $j\ge 0$ we get the generating function
\begin{equation} \label{eq:stirling1:gf:partial}
  \frac{(\log(1+t))^k}{k!} = \sum_{k = j}^\infty s(k,j) \frac{t^k}{k!} \qquad (|t| < 1).
\end{equation}

The \emph{Stirling numbers of the second kind} are the coefficients that arise when one expresses $x^k$ as a linear combination of the falling factorials $x^{\underline{0}}, x^{\underline{1}}, \ldots, x^{\underline{k}}$, that is,
\[x^k = \sum_{j=0}^k \stirlingtwo{k}{j} x^{\underline{j}},\]
where $k$ is a non-negative integer. From the previous definition one obtains
\begin{align*}
  \stirlingtwo{k}{0} & = 0 \text{ for every } k \ge 1, \\
  \stirlingtwo{k}{j} & = 0 \text{ for every } j > k \ge 0,\\
  \stirlingtwo{k}{k} & = 1.
\end{align*}
It is not difficult to establish the recurrence relation
\[\stirlingtwo{k}{j} = \stirlingtwo{k-1}{j-1} + k\stirlingtwo{k-1}{j},\]
valid for every $k$ and $j$. As a consequence, $\stirlingtwo{k}{j}$ is the number of partitions of a set with $k$ objects into $j$ non-empty sets (as these numbers
are easily seen to also satisfy this recurrence and the same boundary conditions). This is the reason why these quantities are sometimes known as Stirling set numbers. From this observation we conclude that $\stirlingtwo{k}{j} \ge 1$, $\stirlingtwo{k}{1} = 1$ and $\stirlingtwo{k}{2} = 2^{k-1}-1$ for 
all integers $k$ and $j$ with $1\le j \le k$. Stirling numbers of the first and second kind are closely related. From \eqref{eq:stirling1} it follows that
\[ x^{\underline{k}} = \sum_{t = 0}^k s(k,t) \sum_{j = 0}^k \stirlingtwo{t}{j} x^{\underline{j}} = \sum_{j = 0}^k  x^{\underline{j}}  \sum_{t = 0}^k s(k,t) \stirlingtwo{t}{j}, \]
or, equivalently,
\begin{equation} \label{eq:stirling:matrix}
    \sum_{t = 1}^k s(k,t) \stirlingtwo{t}{j} = \delta_{k,j}
\end{equation}
for all positive integers $k$ and $j$, where $\delta_{k,j}$ is $1$ when $k = j$ and $0$ otherwise. 

Let $n$ be a positive integer. Let $C$ and $S$ be the $n \times n$ dimensional matrices with entries are 
$C_{kj} = s(k,j),$ respectively $S_{kj} = \stirlingtwo{k}{j}.$ The equalities given in \eqref{eq:stirling:matrix} can be rewritten as 
the matrix identity $C S = I_n$, 
where $I_n$ is the 
$n \times n$ identity matrix. Hence we have $S C = I_n$, that is,
\begin{equation} \label{eq:stirling:matrix:2}
    \sum_{t = 1}^k \stirlingtwo{k}{t} s(t,j) = \delta_{k,j}
\end{equation}
for all positive integers $k$ and $j$.

\subsection{Fa\`a di Bruno's formula and Bell polynomials \cite{faadibruno}}
In order to define Bell polynomials and  
formulate Fa\`a di Bruno's\footnote{Or rather the Blessed Chevalier Fa\`a di Bruno if we 
take the distinctions conferred upon him by both the Roman Catholic Church and the Military into account.} formula, the notion of a partition of 
an integer is needed. A partition of a positive integer $k$ can be identified with a sequence $\{\lambda_j\}_{j=1}^{\infty}$ of  non-negative integers satisfying $\sum_{j}j \lambda_j=k$.  Without loss of generality we can write a partition, $\lambda$, of $k$ as $\lambda=(\lambda_1, \ldots, \lambda_k)$, where $\lambda_j \ge 0$ and $\sum_{j = 1}^k j \lambda_j = k$. The set of all partitions of $k$ will be denoted by ${\mathcal P}(k)$.

For every positive integer $k$ and $1 \le j \le k$ the \emph{partial Bell polynomial} ${\mathcal B}_{k,j}(x)$ is defined as 
\begin{equation*} \label{eq:bell}
  \mathcal{B}_{k,j}(x_{1},x_{2},\dots ,x_{k-j+1})=\sum \frac{k!}{\lambda_{1}!\lambda_{2}!\cdots \lambda_{k-j+1}!} \left(\frac{x_{1}}{1!}\right)^{\lambda_{1}}\left(\frac{x_{2}}{2!}\right)^{\lambda_{2}}\cdots \left(\frac{x_{k-j+1}}{(k-j+1)!}\right)^{\lambda_{k-j+1}},
\end{equation*}
where the sum is taken over all the partitions $(\lambda_1, \lambda_2, \ldots, \lambda_{k-j+1}, 0, \ldots, 0) \in {\mathcal P}(k)$ such that $\lambda_1 + \lambda_2 + \cdots + \lambda_{k-j+1} = j$, see \cite[Chapter 3]{Comtet}. 
We have, e.g., $\mathcal{B}_{4,2}(x_{1},x_{2},x_{3})=4x_1x_3+3x_2^2$.

Partial Bell polynomials can also be defined in terms of set partitions. This approach is used in \cite{faadibruno}, where Fa\`a di Bruno's formula is studied. This formula is a generalization of the chain rule for higher order derivatives. It can be formulated in several ways although in our opinion its simplest statement is in terms of partial Bell polynomials.

\begin{lem}[Fa\`a di Bruno's formula] Let $k$ be a positive integer, $f$ and $g$ be
$k$ times differentiable functions on intervals 
$I$, respectively $J$ in $\mathbb R$. Set $x \in J$ such that $g(x) \in I$. Then 
\begin{align*}
(f(g))^{(k)}(x) & =  \sum_{j=1}^k f^{(j)}(g(x)) \mathcal{B}_{k,j}\left(g'(x),g''(x),\ldots,g^{(k-j+1)}(x)\right) \\
& = k! \sum_{(\lambda_1, \ldots, \lambda_k) \in \mathcal{P}(k)} f^{(\lambda_1 + \cdots + \lambda_k)}(g(x)) \prod_{j = 1}^k \frac{1}{\lambda_j !} \left(\frac{g^{(j)}(x)}{j!}\right)^{\lambda_j}.
\end{align*}
\end{lem}
The $k^{th}$ \emph{complete Bell polynomial} is defined as 
\[ \mathcal{B}_k(x_1, \ldots, x_k) = \sum_{j = 1}^k B_{k,j}(x_1 \ldots, x_{k-j+1}). \]
Let $h$ be a $k$ times differentiable function on an interval $I$ and let $x \in I$ such that $h(x) > 0$. By applying Fa\`a di Bruno's formula with $f = \exp$ and $g = \log h$ we obtain
\begin{equation} \label{eq:faa-di-bruno:exp}
h^{(k)}(x) = h(x)\, \mathcal{B}_k\left((\log h)'(x),(\log h)''(x),\ldots,(\log h)^{(k)}(x)\right).
\end{equation}
Complete Bells polynomials obey the recurrence relation
\begin{equation} \label{eq:bell:recurrence}
  \mathcal{B}_{k}(x_1, \ldots, x_k)=\sum _{j=1}^{k} \binom{k-1}{j-1} \mathcal{B}_{k-j}(x_1, \ldots ,x_{k-j})x_j   
\end{equation}
for every $k \ge 1$, where $\mathcal{B}_0$ is defined as $1$. This gives an effective method to compute $\mathcal{B}_{k}$.

\subsection{Numerical semigroups \cite{MoreeAMS,Rosa}}
\label{sec:pre:ns}
A \emph{numerical semigroup} $S$ is a submonoid of $\mathbb N$ (the set of nonnegative integers) under addition, with finite complement $\mathcal G$ in $\mathbb N$. The cardinality of $\mathcal G$ is denoted by
$\mathrm{g}(S)$ and called the \emph{genus} of $S$. The \emph{Frobenius number} of $S$, denoted by $\mathrm{F}(S)$, is the largest element of $\mathcal{G}$. A numerical semigroup has a unique minimal set of generators 
$\{n_1,\ldots,n_e\}$. We write $S=\langle n_1,\ldots,n_e \rangle$. The integer $\mathrm{e}(S) = e$ is the \emph{embedding dimension} of $S$.
To a numerical semigroup $S$ we can associate 
\[\mathrm{P}_S(x)=1+(x-1)\sum_{g \in \mathcal G}x^g,\] 
its {\it semigroup polynomial}.
Note that $\mathrm{P}_S(x)$ is a monic polynomial of degree 
$\mathrm{F}(S)+1$.
A numerical semigroup is called symmetric if 
$S\cup (\mathrm{F}(S)-S) = \mathbb Z,$ 
with $\mathrm{F}(S)-S=\{\mathrm{F}(S)-s:s\in S\}.$
In terms of
$\mathrm{P}_S$ this is equivalent with $\mathrm{P}_S$ being self-reciprocal, 
cf. \cite{MoreeAMS}.

In \cite{CGM} the notion of a \emph{cyclotomic numerical semigroup} is introduced
and studied. This is a numerical semigroup such that its associated semigroup 
polynomial is a Kronecker polynomial.
Cyclotomic numerical semigroups have some interesting properties. Indeed in \cite{CGM} the concept of cyclotomic numerical semigroup is conjectured to coincide with that of complete intersection, the latter being the topic of many publications (see \cite{isolated} and the references given therein).

\subsection{The logarithmic derivative}

Given a rational function $f$ that is a product of
finitely many rational functions $g_i,$ its derivative 
outside the set of poles and zeros of
the $g_i$ satisfies
$f'=f\sum_i g_i'/{g_i}.$
It follows from this that
$$\frac{f'}{f}=\sum_i \frac{g_i'}{g_i}$$
outside the set of poles and 
zeros of
the $g_i$. For $k\ge 1$ by 
$(\log f)^{(k)}$ we 
mean $(f'/f)^{(k-1)}$. The upshot
is that in our logarithmic derivatives no logarithms are involved. This saves us from 
various considerations involving 
domain, range and 
multi-valuedness that normally
arise if the word `logarithm' occurs. 
With this in mind the reader can without 
problem infer a range of validity of the expressions 
appearing in this paper involving the logarithmic derivative 
and we will not discuss this topic again.

\section{The $k^{th}$ logarithmic derivative of $\Phi_n$} \label{sec:phi-log-deriv}

In this section we describe a procedure to derive formulas for the $k^{th}$ logarithmic derivative of $\Phi_n$ at any given point. We apply this procedure to find the logarithmic derivatives of $\Phi_n$ at $0$ and $\pm1$, thus reproving results of Lehmer \cite{Lehmer} and Nicol \cite{Nicol}. 
From  equality \eqref{phimoebius} we get
\begin{equation} \label{eq:procedure:1}
  (\log \Phi_n)^{(k)}(x) = \sum_{d \mid n} \mu\left(\frac{n}{d}\right) (\log q_d)^{(k)}(x),
\end{equation}
where $k$ and $n$ are positive integers with $n \ge 2$ and $q_d(x) = 1 - x^d$.
For every $d \ge 1$ set $p_d(x) = q_d(x) / (1-x)$.  
Since $\sum_{d \mid n} \mu(d) = 0$ for $n \ge 2$, it follows that 
\begin{equation}\label{eq:kth-deriv:phi}
  (\log \Phi_n)^{(k)}(x) = \sum_{d \mid n} \mu\left(\frac{n}{d}\right) (\log p_d)^{(k)}(x)
\end{equation}
for all integers $n \ge 2$ and $k \ge 1$. As a consequence, if one computes the logarithmic derivatives of $q_d(x)$ or $p_d(x)$, then one obtains a formula for $(\log \Phi_n)^{(k)}(x)$.

\subsection{Ramanujan sums and the $k^{th}$ logarithmic derivative of $\Phi_n$ at $0$} \label{sec:kth-deriv:0}
In view of Taylor's theorem, proving Nicol's identity \eqref{eq:nicol} amounts to showing Proposition \ref{prop:nicol}. 
\begin{prop} \label{prop:nicol}
  Let $k\ge 1$ and $n \ge 2$ be integers. Then 
  \[ (\log \Phi_n)^{(k)}(0) = - (k-1)! \, r_k(n). \]
\end{prop}
\begin{proof}
By taking the $k^{th}$ logarithmic  derivative of both sides of \eqref{definitie} we obtain 
  \begin{equation} \label{eq:kth-log-deriv:roots}
    (\log \Phi_n)^{(k)}(x) = \sum_{1\le j\le n,\,(j,n)=1} \frac{(-1)^{k-1}(k-1)!}{(x - \zeta_n^j)^k}.
  \end{equation}
The result follows on settting $x=0$ in this expression and noting that $\overline{r_k(n)}=r_k(n)$.
  \end{proof}
Recall that around $x = 0$ we have $\log(1 - x) = -\sum_{j = 1}^\infty  x^j / j$ and, therefore, the $k^{th}$ derivative of $\log(1-x^d)$ at $0$ is $-d (k-1)!$ if $d \mid k$ and $0$ otherwise. This observation in combination with \eqref{eq:procedure:1} yields 
\begin{equation} \label{eq:nicol:2}
   (\log \Phi_n)^{(k)}(0) = 
   -(k-1)!\sum_{d \mid (k,n)} \mu\left(\frac{n}{d}\right) d \qquad (n \ge 2).
\end{equation}
We observe that Kluyver's formula \eqref{eq:kluyver} follows on comparing \eqref{eq:nicol:2} with Proposition \ref{prop:nicol}, which highlights the deep connection between Ramanujan sums and cyclotomic polynomials.

In order to prove his formula Nicol first showed that
\begin{equation}
\label{eq:nicol2}
  \sum_{k = 1}^n r_k(n) x^{k-1} = (x^n - 1) \frac{\Phi_n'(x)}{\Phi_n(x)} \qquad (n \ge 1).
\end{equation}
From this equality one can derive \eqref{eq:nicol} by integration, see \cite[Corollary 3.2]{Nicol} for more details.
Reversely, one can easily deduce \eqref{eq:nicol2} from \eqref{eq:nicol}. Namely,
we note that for $n \ge 2$ differentiation of both sides
of \eqref{eq:nicol} yields 
\begin{equation}
\label{eq:above}
-\frac{\Phi'_n(x)}{\Phi_n(x)} = \sum_{k=1}^{\infty}r_{k}(n)x^{k-1} 
                              = \sum_{k=1}^{n}r_{k}(n)x^{k-1}(1+x^n+x^{2n}+\cdots)
                              = \frac{1}{1-x^n}\sum_{k=1}^{n}r_{k}(n)x^{k-1},
\end{equation}
where we used that $r_k(n) = r_{n+k}(n)$ for every positive integer $k$.
Therefore, \eqref{eq:nicol2} holds true for those $x$ with $|x|<1$.
As both sides of \eqref{eq:nicol2} are polynomials that
agree for all $|x|<1$, they agree for all $x$. Note that the case $n=1$ is trivial.

The reader can find other arguably longer proofs of \eqref{eq:nicol} and \eqref{eq:nicol2} in \cite[Theorem 1]{Toth}.

\subsection{The $k^{th}$ logarithmic derivative of $\Phi_n$ at $1$}  \label{sec:kth-log-deriv:1}

In order to prove Theorem \ref{thm:kth-deriv:1} we study the derivatives of $\log \Phi_n$ at $1$. For any integers $k\ge1$ and $n \ge 2$ let us consider the sum 
\[ s_k(n) = \sum_{1\le j\le n,\,(j,n)=1} (\zeta_n^j - 1)^{-k}, \]
which is of similar nature to that of Ramanujan. The sum $s_k(n)$ was studied by Lehmer in \cite[Section 5]{Lehmer} and it plays an essential role in his proof of Theorem \ref{thm:kth-deriv:1}. 
On setting $x=1$ in \eqref{eq:kth-log-deriv:roots} we obtain
\begin{equation} \label{eq:sk}
  (\log \Phi_n)^{(k)}(1) = - (k-1)! s_k(n).
\end{equation}
Our proof uses the identity \eqref{eq:kth-deriv:phi} and hence it amounts to finding the $k^{th}$ order logarithmic derivative of $p_n(x) = 1+x+\cdots+x^{n-1}$ at $x=1$. By applying the same arguments as in the proof of Proposition \ref{prop:nicol} we see that 
\begin{equation} \label{eq:pn}
  (\log p_n)^{(k)}(1) = - (k-1)! \sigma_k(n),
\end{equation}
where
\[ \sigma_k(n)=\sum_{j = 1}^{n-1} (\zeta_n^j -1)^{-k}. \]
Not surprisingly $\sigma_k$ and $s_k$ are closely connected. Putting $s_k(1)=0$ we find by \eqref{eq:kth-deriv:phi} that
\[ s_k(n)=\sum_{d\mid n}\mu(d)\sigma_k(\frac{n}{d}), \]
and, by M\"obius inversion, that
\[ \sigma_k(n)=\sum_{d\mid n}s_k(d) \]
The values $\sigma_k(n)$ were studied by Lehmer in \cite[Section 4]{Lehmer}, where he proved Theorem \ref{thm:gessel}.
Later Gessel published a significantly shorter proof of this result \cite[Theorem 2.1]{Gessel}. His aim
was to show that for any fixed $k$ the arithmetic function $\sigma_k(n)$ is a polynomial in $n$. Here we present a version of Gessel's proof for completeness.
\begin{Thm}[{\cite[Lemma 4]{Lehmer}}] \label{thm:gessel}
  Let $k \ge 1$ and $n\ge 2$ be integers. Then
  \[ -(k-1)! \sigma_k(n) = \sum_{j = 1}^k s(k,j) \frac{B_j^+}{j}(n^j -1). \]
 \end{Thm}
\begin{proof}
  The proof consists in finding a series expansion of $(\log p_n)(x)$ at $x=1$. First, note that
  \[  \frac{\partial }{\partial y} \log\Big(\frac{e^y-1}{y}\Big) = \frac{1}{y} \left(\frac{y e^y}{e^y - 1} - 1\right) = \sum_{j = 0}^\infty \frac{B_{j+1}^+}{(j+1)!} y^j \qquad (|y| < \pi),\]
  where we used \eqref{eq:bernoulli-pol:gen}. Therefore, we obtain
  \[ \log\Big(\frac{e^{ny} -1}{n(e^y -1)}\Big) = \log\Big(\frac{e^{ny} - 1}{n y}\Big) - \log\Big(\frac{e^{y} - 1}{y}\Big) = \sum_{j = 1}^\infty \frac{B_{j}^+}{j} \frac{y^j}{j!} (n^j - 1) \qquad (|y| < \pi).\]
  That is, we have
  \[ (\log p_n)(x) = \log n + \sum_{j = 1}^\infty \frac{B_{j}^+}{j} \frac{(\log x)^j}{j!} (n^j - 1) \qquad (e^{-\pi} < x < e^\pi).\]
  From equation \eqref{eq:stirling1:gf:partial} it follows that
  \begin{align*}
  (\log p_n)(x) & = \log n + \sum_{j = 1}^\infty \frac{B_{j}^+}{j} (n^j - 1) \sum_{k = j}^\infty s(k,j) \frac{(x-1)^k}{k!} \\
                & = \log n + \sum_{k = 1}^\infty \frac{(x-1)^k}{k!} \sum_{j = 1}^k \frac{B_{j}^+}{j} s(k,j) (n^j - 1) \qquad (e^{-\pi} < x < 2).
  \end{align*}
  The proof is completed on recalling \eqref{eq:pn}.
\end{proof}

For every $1 \le j \le k$ write 
\[-(k-1)!\sigma_k(n) = \sum_{j = 0}^k c_{k,j} n^j\text{ with }c_{k,0} = -\sum_{j = 1}^k \frac{B_j^+ s(k,j)}{j} \quad \text{and} \quad c_{k,j} = \frac{B_j^+ s(k,j)}{j}.\]
 Table \ref{table:kth-log-deriv:1} gives some of the coefficients $c_{k,j}$, which can be computed by applying the recursive definitions of $B_j^+$ and $s(k,j)$.

\begin{table}[H]
\begin{tabular}{|c|ccccccccc|}
\hline
$k$ & $c_{k,0}$ & $c_{k,1}$ & $c_{k,2}$ & $c_{k,3}$ & $c_{k,4}$ & $c_{k,5}$ & $c_{k,6}$ & $c_{k,7}$ & $c_{k,8}$\\
\hline
$1$ & $-1/2$ & $1/2$ & & & & & & & \\
\hline
$2$ & $5/12$ & $-1/2$ & $1/12$  & & & & & & \\
\hline
$3$ & -$3/4$ & $1$ & $-1/4$ & $0$ & & & & & \\
\hline
$4$ & $251/120$ & $-3$ & $11/12$ & $0$ & $-1/120$  & & & & \\
\hline
$5$ & $-95/12$ & $12$ & $-25/6$ & $0$ & $1/12$ & $0$ & & & \\
\hline
$6$ & $19087/504$ & $-60$ & $137/6$ & $0$ & $-17/24$ & $0$ & $1/252$ & & \\
\hline
$7$ & $-5257/24$ & $360$ & $-147$ & $0$ & $49/8$ & $0$ & $-1/12$ & 0 & \\
\hline
$8$ & $1070017/720$ & $-2520$ & $1089$ & $0$ & $-6769/120$ & $0$ & $23/18$ & $0$ & $-1/240$ \\
\hline
\end{tabular}

\caption{The coefficients $c_{k,j}$ for small values of $k$.}
\label{table:kth-log-deriv:1}
\end{table}

Next we determine $(\log \Phi_n)^{(k)}(1)$ by applying some properties of Jordan totient functions.

\begin{Thm}\label{thm:kth-log-deriv:1}
  Let $k \ge 1$ and $n\ge 2$ be integers. Then
  \[(\log \Phi_n)^{(k)}(1) = \sum_{j=1}^{k} \frac{B_j^+ s(k,j)}{j} J_j(n).\]
\end{Thm}
\begin{proof}
By evaluating \eqref{eq:kth-deriv:phi} at $1$ and 
noting that $\sum_{d\mid n}\mu(n/d)=0$ 
we infer that
 \begin{align*}
(\log \Phi_n)^{(k)}(1) & = - \sum_{d \mid n} \mu\left(\frac{n}{d}\right) (k-1)! \sigma_k(n) = \sum_{d \mid n} \mu\left(\frac{n}{d}\right) \sum_{j=1}^{k} c_{k,j} n^j \\ 
                        & = \sum_{j=1}^{k} c_{k,j} \sum_{d \mid n} \mu\left(\frac{n}{d}\right) n^j  = \sum_{j=1}^{k} \frac{B_j^+ s(k,j)}{j} J_j(n). \qedhere
\end{align*}
\end{proof}

Lehmer's proof of Theorem \ref{thm:kth-deriv:1} hinges upon the equality 
\[ 
  s_k(n) = \frac{(-1)^{k}}{2} \varphi(n) - \frac{1}{(k-1)!}\sum_{j=2}^{k} \frac{B_j^- s(k,j)}{j} J_j(n), 
\]
which in our notation can be rewritten as
\begin{equation} \label{eq:lehmer}
  s_k(n) = - \frac{1}{(k-1)!}\sum_{j=1}^{k} \frac{B_j^+ s(k,j)}{j} J_j(n).
\end{equation}
Note that, in view of \eqref{eq:sk}, the previous equation \eqref{eq:lehmer} is equivalent to Theorem \ref{thm:kth-log-deriv:1}.

Now on applying Fa\`a di Bruno's formula 
\eqref{eq:faa-di-bruno:exp} to $\Phi_n = \exp(\log \Phi_n),$ 
we obtain the following result.
\begin{Thm} \label{thm:phi:1}
Let $k \ge 0$ and $n\ge 2$ be integers. 
  Then $$
\label{eq:phi:faa}
  \frac{\Phi_n^{(k)}(1)}{\Phi_n(1)}= \mathcal{B}_k\Big(\sum_{j=1}^{1} \frac{B_j^+ s(1,j)}{j} J_j(n), \dots, \sum_{j=1}^{k} \frac{B_j^+ s(k,j)}{j} J_j(n)\Big).
$$
\end{Thm}
On writing out the Bell polynomial explicitly we obtain Theorem \ref{thm:kth-deriv:1},
which was proven by Lehmer without employing Fa\`a di Bruno's formula. If we use the recursive definition of Bell polynomials instead, see \eqref{eq:bell:recurrence}, we obtain the following novel result, which establishes a recurrence relation between the derivatives of $\Phi_n$ 
of different order at $1$.
\begin{Thm} \label{thm:kth-deriv:1:recurrence}
Let $k \ge 0$ and $n\ge 2$ be integers. Then
$$
\Phi_n^{(k)}(1)  = -(k-1)! \sum_{j = 1}^k \frac{\Phi_n^{(j)}(1)}{j!} \sigma_{k-j}(n)  = \sum_{j = 1}^k \binom{k-1}{j-1} \Phi_n^{(k-j)}(1) \sum_{t=1}^{j} \frac{B_t^+ s(j,t)}{t} J_t(n).
$$
\end{Thm}
\begin{proof}
We have
\begin{align*}
  \Phi_n^{(k)}(1) & = \Phi_n(1) \mathcal{B}_{k}(-\sigma_1(n), \dots, -(k-1)! \sigma_k(n)) \\
         & = - \Phi_n(1) \sum _{j=1}^{k}(j-1)! \sigma_j(n) \binom{k-1}{j-1} \mathcal{B}_{k-j}(-\sigma_1(n), \dots, -(k-1-j)! \sigma_{k-j}(n)) \\
         & = - (k-1)! \sum _{j=1}^{k} \sigma_j(n) \frac{\Phi_n^{(k-j)}(1)}{(k-j)!}=-(k-1)! \sum_{j = 1}^k \frac{\Phi_n^{(j)}(1)}{j!} \sigma_{k-j}(n).
\end{align*}
The proof is completed on invoking Theorem  \ref{thm:gessel}.
\end{proof}

For $k = 1$ and $n\ge 2$ we obtain 
\[\frac{\Phi_n'(1)}{\Phi_n(1)} = (\log \Phi_n)'(1) = c_{1,1} \varphi(n) = \frac{\varphi(n)}{2},\]
which recovers part of Corollary \ref{phin1:deriv}. 
For $k=2$ we get identity \eqref{formula:k=2}. For $k=3$ the obtained identity can also be expressed in terms of $\varphi$ and $\Psi$. Recall that $\mathcal{B}_3(x_1, x_2, x_3) = x_1^3 + 3x_1 x_2 + x_3$. We have
\begin{equation} \label{formula:k=3}
\begin{aligned}
\frac{\Phi_n^{'''}(1)}{\Phi_n(1)} & = \mathcal{B}_3\Big(\frac{\varphi(n)}{2}, \frac{\varphi(n) \Psi(n)}{12} - \frac{\varphi(n)}{2}, \varphi(n)-\frac{\varphi(n) \Psi(n)}{4}\Big) \\
& = \frac{1}{8}\varphi(n)^3+\frac{1}{8}\varphi(n)^2\Psi(n)-\frac{3}{4}\varphi(n)^2-\frac{1}{4}\varphi(n)\Psi(n)+1.
\end{aligned}
\end{equation}
With increasing $k$ the formulas produced by Theorem \ref{thm:phi:1} 
become more and more cumbersome to write down.

The Schwarzian derivative of a holomorphic function $f$ of one complex variable $z$ (cf. \cite{ovsienko}) is defined 
as $\mathrm{S}(f)(z) = f'''(z) / f'(z) - \frac{3}{2} \left(f''(z) / f'(z)\right)^2$. We can apply Theorem \ref{thm:phi:1} to obtain a formula for $\mathrm{S}(\Phi_n)(1).$
\begin{cor}
For $n \ge 2$ we have
\[\mathrm{S}(\Phi_n)(1) = -\frac{\varphi(n)^2}{8} - \frac{\Psi(n)^2}{24} + \frac{1}{2}.\]
\end{cor}
\begin{proof}
Let $n \ge 2$ be an arbitrary integer.
From the identity 
$\Phi_n'(1)/\Phi_n(1)=\varphi(n)/2$ and the formulas \eqref{formula:k=2}
and \eqref{formula:k=3} for $\Phi_n''(1) / \Phi_n(1),$ 
respectively $\Phi_n'''(1) / \Phi_n(1),$  one finds that
$$\frac{\Phi^{''}_n(1)}{\Phi'_n(1)}=\frac{1}{2}
\varphi(n)+\frac{1}{6}\Psi(n)-1,$$
respectively
$$\frac{\Phi^{'''}_n(1)}{\Phi'_n(1)}=\frac{1}{4}
\varphi(n)^2+\frac{1}{4}\varphi(n)\Psi(n)-\frac{3}{2}\varphi(n)-\frac{\Psi(n)}{2}+2.$$
The proof follows 
from the latter two identities
and the definition of the 
Schwarzian.
\end{proof}
This expression for the Schwarzian is remarkably compact. Indeed, on changing the 
$3/2$ in the definition of the 
Schwarzian to any other number, the
number of terms in the resulting expression becomes larger than three.

\subsection{The $k^{th}$ logarithmic derivative of $\Phi_n$ at $- 1$}

In this section we use Theorem \ref{thm:kth-log-deriv:1} to evaluate the derivatives of $\log \Phi_n$ at $-1$. First, we need the following lemma. A less general version of this lemma appears in Lehmer's work \cite[Theorem 5]{Lehmer}.

\begin{lem} \label{lem:kth-deriv-x}
  Let $g : \mathbb{R}^+ \to \mathbb{R}$ be a $k$ times differentiable function. 
  Put $h_n=g(\Phi_n)$.
  For every $k \ge 0$ and $n \ge 3$ we have 
  \[ h_n^{(k)}(x) = (-1)^k h_{n \alpha_n}^{(k)}(-x), \]
  with
  \begin{equation} \label{alpha}
    \alpha_n = \begin{cases}2 & \hbox{if } n \hbox{ is odd};\\ 1/2 & \hbox{if } 2 \parallel n;\\ 1 & \hbox{otherwise.}\end{cases}
  \end{equation}
\end{lem}
\begin{proof}
  Note that $\Phi_n(x) > 0$ for every $x \in \mathbb{R}$ and $n \ge 3$. We distinguish three cases:
  \begin{enumerate}
    \item $2 \nmid n$. We have $\Phi_n(x) = \Phi_{2n}(-x)$. Hence $h_n^{(k)}(x) = (-1)^{k}h_{2n}^{(k)}(-x)$ and $\alpha_n=2$.
    \item $2 \parallel n$. We have $\Phi_n(x) = \Phi_{n/2}(-x)$. Hence $h_n^{(k)}(x) = (-1)^{k}h_{n/2}^{(k)}(-x)$ and $\alpha_n=1/2$.
    \item $4 \mid n$. We have $\Phi_n(x) = \Phi_n(-x)$. Hence $h_n^{(k)}(x) = (-1)^kh_n^{(k)}(-x)$ and $\alpha_n=1$. \qedhere
  \end{enumerate}
\end{proof}
 
 Note that $(\log \Phi_1)^{(k)}(-1) = - (k-1)! / 2^k = (-1)^k(\log \Phi_2)^{(k)}(1)$. One can use this observation and Lemma \ref{lem:kth-deriv-x} with $g = \log$ to easily compute $(\log \Phi_n)^{(k)}(-1)$ from Theorem \ref{thm:kth-log-deriv:1}.
\begin{Thm} \label{thm:kth-log-deriv:-1}
  Let $k$ and $n$ be positive integers with $n \ne 2$. Then
  \[ (\log \Phi_n)^{(k)}(-1) = (-1)^{k-1} (k-1)! s_k(n \alpha_n) =  (-1)^k \sum_{j=1}^{k} \frac{B_j^+ s(k,j)}{j} J_j(n \alpha_n), \]
  with $\alpha_n$ as in \eqref{alpha}.
\end{Thm}

Now on applying Fa\`a di Bruno's formula 
\eqref{eq:faa-di-bruno:exp} to $\Phi_n = \exp(\log \Phi_n)$ for $n \ge 3$ and $-\Phi_1 = \exp(\log(-\Phi_1))$ 
we obtain the following 
analogue of Theorem \ref{thm:phi:1}.

\begin{Thm} \label{thm:kth-deriv:-1}
  Let $k$ and $n$ be positive integers with $n \ne 2$. Then
\[ \frac{\Phi_n^{(k)}(-1)}{\Phi_n(-1)}= \mathcal{B}_k\Big(-\sum_{j=1}^{1} \frac{B_j^+ s(1,j)}{j} J_j(n \alpha_n), \ldots, (-1)^k\sum_{j=1}^{k} \frac{B_j^+ s(k,j)}{j} J_j(n \alpha_n)\Big), \]
  with $\alpha_n$ as in \eqref{alpha}.
\end{Thm}
Taking $k = 2$ in this result gives the following corollary.
\begin{cor} 
  For $n=1$ or $n > 2$ we have
  \[ \frac{\Phi_n''(-1)}{\Phi_n(-1)} =  \frac{\varphi(n)}{4} \left(\varphi(n) + \frac{1}{3}\Psi(n \alpha_n) - 2\right). \]
\end{cor}
\begin{proof}
 Since $\mathcal{B}_2(x_1, x_2) = x_1^2 + x_2$, we have
\[\frac{\Phi_n''(-1)} {\Phi_n(-1)} = \frac{1}{4}\varphi(n \alpha_n)^2 - \frac{1}{2}\varphi(n \alpha_n) + \frac{1}{12}J_1(n \alpha_n) = \frac{\varphi(n)}{4} \left( \varphi(n) + \frac{1}{3}\Psi(n \alpha_n) -2 \right), \] 
where we used that $\varphi(n \alpha_n) = \varphi(n)$.
\end{proof}
It is also possible to give a  recurrence relation for $\Phi_n^{(k)}(-1)$ as we did in Theorem \ref{thm:kth-deriv:1:recurrence}. Doing so we leave to the interested reader.

\subsection{Average behavior of higher order derivatives of 
cyclotomic polynomials} 
Let $f(x)\in \mathbb Z[x]$ be a polynomial and let $\deg f$ denote its degree with respect to $x$. 
For any complex number $z$ such that $f(z)\ne 0$, we define
\begin{equation}
N^{(k)}(z)=\frac{1}{(\deg f)^k}\frac{f^{(k)}(z)}{f(z)}.
\end{equation}
In case 
$f(x)\in \mathbb Z_{\ge 0}[x]$, $z>0$ is real and $f(z)\ne 0$,
it is easy to show by induction that $N^{(k)}(z) \le 1.$ 
This property suggests to 
call $N^{(k)}(z)$ the normalized $k^{th}$ derivative of $f$ at $z$ and
leads to the following problem.
\begin{Problem} Let $z$ be given. Let $\mathcal F$ be an 
infinite family of polynomials $f$ with $f(z)\ne 0$. Study the 
average behavior and value distribution of  $N^{(k)}(z)$ in the family $\mathcal F$.
\end{Problem}
Moree et al. \cite{MSSS} consider this problem in depth. They take $\mathcal F$ to be the family
of cyclotomic polynomials and $z \in \{-1,0,1\}$ and make crucial use of
the results in this paper to express the quantities under
consideration as linear combinations of Jordan totient quotients.
E.g., Theorem \ref{thm:kth-deriv:1} expresses 
${\Phi_n^{(k)}(1)}/{\Phi_n(1)}$ as a $\mathbb Q$-linear 
combination of Jordan totient products of weight 
not exceeding $k,$ which
 has as a consequence that the average 
$$\lim_{x\rightarrow \infty}\frac{1}{x}\sum_{n\le x}
\frac{1}{\varphi(n)^k}\frac{\Phi_n^{(k)}(1)}{\Phi_n(1)}$$
exists, as the argument of the sum is a finite $\mathbb Q$-linear combination of Jordan totient quotients of non-negative weight, 
which by \cite{MSSS} are each constant on average.
Likewise, several other quantities in this paper, after appropriate normalization, can be shown
to be constant on average (actually usually far more precise results than being constant on average
can be formulated).

\section{Formulas for the coefficients of cyclotomic polynomials}
\label{sec:coefficientformulas}

In this section we gather all formulas known to us for the coefficients $a_n(k)$ with $n$ arbitrary. 
There are also identities for restricted subsets of the integers,
such as binary and ternary integers (these are composed of 
precisely two, respectively precisely three distinct odd prime factors), 
but we will not consider them here. 
The binary formula is folklore, see, e.g., \cite{LaLe}.
A ternary formula that found a lot of applications is
due to Kaplan \cite{Kaplan}. 

Let $n>1$ be arbitrary. Note that $\Phi_n^{(k)}(0)=k!a_n(k)$ for $k\le \varphi(n)$. We put
$a_n(k)=0$ for $k>\varphi(n)$ so that, in light of \eqref{phimoebius}, we can write 
\begin{equation}
  \label{infproduct}
  \prod_{d=1}^{\infty}(1-x^d)^{\mu(n/d)}=\sum_{j=0}^{\infty}a_n(j)x^j,
\end{equation}
where we define $\mu(t) = 0$ if $t$ is not an integer. Now note that, for $|x|<1$, we have
\begin{equation*}
  \label{thanga2}
  (1-x^d)^{\mu({n/d})}=1-\mu(n/d)x^d+\frac{1}{2}\mu(n/d)\left(\mu(n/d)-1\right)\sum_{j=2}^{\infty}x^{jd}.
\end{equation*}
We infer that
\begin{equation} \label{lehmertje}
\prod_{d=1}^{\infty}(1-x^d)^{\mu({n/d})}=\prod_{d=1}^{\infty}\Bigg(1-\mu(n/d)x^d+
\frac{1}{2}\mu(n/d)\left(\mu(n/d)-1\right)\sum_{j=2}^{\infty}x^{jd}\Bigg),
\end{equation}

From \eqref{infproduct} and \eqref{lehmertje} we see
that $a_n(k)$ is just the coefficient of $x^k$ in the right-hand side of (\ref{lehmertje}). Therefore, we obtain the already well-known values
\begin{equation} \label{4cases}
  \begin{cases}
    \Phi_n(0)=1; \\
    \Phi'_n(0)=-\mu(n); \\
    \Phi''_n(0)=\mu(n)^2-\mu(n)-2\mu(n/2); \\
    \Phi'''_n(0)=3\mu(n)^2-3\mu(n)+6\mu(n/2)\mu(n)-6\mu(n/3).
  \end{cases}
\end{equation}

One can easily check that
\begin{equation}
  \label{veertien}
  \binom{\mu(n/d)}{j} := \frac{\mu(n/d)^{\underline{j}}}{j!} = \begin{cases}1 & \text{if~}j=0;\cr
  \mu(n/d) &\text{if~}j=1;\cr
  (-1)^j \mu(n/d)(\mu(n/d)-1)/2 &\text{if~}j\ge 2.\end{cases}
\end{equation}
As a consequence, we obtain the following result of H. M\"oller \cite[Satz 2]{HerbertM} (who gave a much
more complicated proof).
\begin{Thm} \label{thm:moller}
  For any non-negative integers $n$ and $k$ with $n \ge 1$ we have
  \begin{equation*}
    a_n(k) = \sum_{(\lambda_1, \ldots, \lambda_k) \in {\mathcal P}(k)} \prod_{j=1}^k (-1)^{\lambda_j} \binom{\mu(n/j)}{\lambda_j}.
  \end{equation*}
\end{Thm}
This result in combination with \eqref{veertien} gives a systematic way of expressing $\Phi_n^{(k)}(0)$ as in \eqref{4cases}. 
Another possible procedure is the one described by Lehmer in \cite[Section 2]{Lehmer}, who used Proposition \ref{prop:nicol} to find a formula for $\Phi_n^{(k)}(0)$ and, thus, $a_n(k)$. 
Let $k \ge 1$ and $n\ge 2$ be integers. For $k = 1$ we obtain 
\[ \Phi_n'(0) = \frac{\Phi_n'(0)}{\Phi_n(0)} = (\log \Phi_n)'(0) = -\mu(n), \]
which recovers one of the equalities given in \eqref{4cases}. Note that $\Phi_n$ is the composition of $\exp$ with $(\log \Phi_n)$ and hence Fa\`a di Bruno's formula allows one to generalize this argument. 
This leads to Lehmer's formula for $a_n(k)$, 
\begin{equation} \label{eq:lehmer:0}
  a_n(k) = \frac{\Phi_n^{(k)}(0)}{k!} = \sum_{(\lambda_1, \ldots, \lambda_k) \in \mathcal{P}(k)} \prod_{j = 1}^k \frac{1}{\lambda_j ! } \left(\frac{-r_j(n)}{j}\right)^{\lambda_j}.
\end{equation}

Lehmer applies the previous formula along with \eqref{eq:kluyver} to find the coefficients $a_n(k)$ for $k \le 10$. Nonetheless, 
in the words of Lehmer, ``this formula is of little immediate help in studying the coefficients''. If one applies Fa\`a di Bruno's formula in terms of Bell polynomials, then the obtained formula is more appealing. 

\begin{Thm} \label{thm:coeffs:0}
  Let $k$ and $n$ be positive integers with $n \ge 2$. Then
  \begin{equation*}
    \Phi_n^{(k)}(0) = \mathcal{B}_k(-r_1(n), \dots, -(k-1)! r_k(n)).
  \end{equation*}
\end{Thm}
Complete Bell polynomials can be computed in several ways, e.g., recursively, see \eqref{eq:bell:recurrence}. Therefore, Theorem \ref{thm:coeffs:0} is more practical than \eqref{eq:lehmer:0}. Indeed, from \eqref{eq:bell:recurrence} it follows that
\begin{align*}
  a_n(k) & = \frac{1}{k!} \mathcal{B}_{k}(-r_1(n), \dots, -(k-1)! r_k(n)) \\
         & = - \frac{1}{k!} \sum _{j=1}^{k}(j-1)! r_j(n) \binom{k-1}{j-1} \mathcal{B}_{k-j}(-r_1(n), \dots, -(k-1-j)! r_{k-j}(n)) \\
         & = - \frac{1}{k}  \sum _{j=1}^{k} r_j(n) a_n(k-j).
\end{align*}
Equivalently, $a_n(k)$ obeys the recurrence relation
\begin{equation*} \label{eq:recurrence}
  a_n(k) = -\frac{1}{k}\sum_{j = 0}^{k-1} a_n(j) r_{k-j}(n) 
\end{equation*}
for every positive integer $k \le \varphi(n)$. This result was noticed first by Grytczuk and Tropak \cite{recurrence}. As Grytczuk and Tropak point out, this recurrence yields an efficient algorithm to compute $a_n(k)$ when $k$ is small. 
Gallot et al.~\cite[Section 9]{GMH} generalized this algorithm so that it also allows one to compute coefficients
of inverse cyclotomic polynomials efficiently.
We refer the reader to \cite{AM} for a survey of algorithms that compute the whole polynomial $\Phi_n$.

Now we use Theorem \ref{thm:phi:1} to find a formula for $a_n(k)$. We have
\begin{equation*}
\begin{aligned}
  \Phi_n(x) & = \sum_{t = 0}^{\varphi(n)} \frac{\Phi_n^{(t)}(1)}{t !} (x - 1)^t =  \sum_{t = 0}^{\varphi(n)} \sum_{k = 0}^t  \frac{\Phi_n^{(t)}(1)}{t!} \binom{t}{k}(-1)^{t-k} x^k \\
            & = \sum_{k = 0}^{\varphi(n)} \frac{x^k}{k!} \sum_{t = k}^{\varphi(n)} \frac{(-1)^{t-k}}{(t-k)!}\Phi_n^{(t)}(1).
\end{aligned}
\end{equation*}
Combining the previous identity with the value of $\Phi_n^{(k)}(1)$ given in Theorem \ref{thm:phi:1} yields the following result.
\begin{Thm} \label{thm:phi:coeffs} For every positive integers $n$ and $k$ with $n \ge 2$ and $k \le \varphi(n)$ we have
  \begin{equation*} 
     a_n(k) = \frac{e^{\Lambda(n)}}{k!} \sum_{t = k}^{\varphi(n)} \frac{(-1)^{t-k}}{(t-k)!} \mathcal{B}_t\Big(\sum_{j=1}^{1} \frac{B_j^+ s(1,j)}{j} J_j(n), \ldots,\sum_{j=1}^{t} \frac{B_j^+ s(t,j)}{j} J_j(n)\Big).
 \end{equation*}
\end{Thm}

An analogous result to Theorem \ref{thm:phi:coeffs} can be derived
by developing $\Phi_n(x)$ as a Taylor series around
$x=-1$ and invoking Theorem \ref{thm:kth-deriv:-1}.

There is also an interpretation in simplicial homology of the cyclotomic coefficients due to Musiker and Reiner \cite{MR}. We will not (re)consider their result here.

\section{Kronecker polynomials} \label{sec:kronecker}

A \emph{Kronecker polynomial} is a monic polynomial with integer coefficients having all 
of its roots on or inside the unit disc. The following result of Kronecker relates Kronecker polynomials with cyclotomic polynomials.

\begin{lem}[Kronecker, 1857; cf. \cite{Kronecker}] \label{kroneckerpols}
If f is a Kronecker polynomial with $f(0) \ne 0$, then all roots of f are actually on the unit circle and f factorizes over the rationals as a product of cyclotomic polynomials.
\end{lem}
By this result  and the fact that cyclotomic polynomials are 
monic and irreducible we can factorize 
a Kronecker polynomial $f(x)$ into irreducibles as 
\begin{equation}
\label{cyclofactor}
f(x) = x^{e_0} \prod_{d \in {\mathcal D}}\Phi_d(x)^{e_d},
\end{equation}
with $e_0 \ge 0$, ${\mathcal D}$ a finite set and each $e_d\ge 1$.

\subsection{Kronecker polynomials at roots of unity}

In \cite{BHMaa} the authors studied the evaluation of cyclotomic polynomials at roots of unity. These results 
have some consequences for Kronecker polynomials as well. We recall the following result from \cite{BHMaa}, which allows us to infer some restrictions on the factorization of $f$ from its values at $\zeta_m$ with $m \in \{1, \ldots, 6\}$.

\begin{lem}
\label{lem:tablecombi}
Let $m\in \{1,2,3,4,6\}$ and $n>m$ be integers. Then
$$|\Phi_n(\xi_m)|=
\begin{cases}
p & \text{if~} n/m=p^k \text{~is a prime power};\\
1 & \text{otherwise}.
\end{cases}
$$
\end{lem}
\begin{cor} \label{lem:factors}
Let $m \in \{1,2,3,4,6\}$. 
Suppose that $f$ is of the form \eqref{cyclofactor}
and, moreover,
satisfies $\min \mathfrak{D}>m$. Then
\[ |f(\xi_m)| = \prod_{\substack{d\in \mathfrak{D} \\ m|d,~\Lambda(d/m)\ne 0} }|\Phi_d(\xi_m)|^{e_d}
              = \exp\left(\sum\nolimits_{d\in \mathfrak{D},~m|d}e_d\Lambda(d/m)\right)\in \mathbb Z_{>0}. \]
\end{cor}
The following result is a reformulation of the latter, but with
$\mathfrak{D}$ assumed to be unknown.
\begin{cor}
\label{reformulation}
Let $f$ be a Kronecker polynomial and $m \in \{1,2,3,4,6\}$. 
Let us also assume that 
$f(\zeta_d) \ne 0$ for every $d\le m$. Then 
$|f(\xi_m)|$ is an integer and each of its 
prime factors $q$
is contributed by a divisor $\Phi_d$ of $f$ with
$d=mq^j$ for some $j\ge 1$. Conversely, if $\Phi_{dq^j}$ divides
$f$, then $q$ divides $|f(\xi_m)|.$
\end{cor}

Corollary \ref{reformulation} will play an important role in our proof of  Theorem \ref{thm:ciolan} (see Lemma \ref{lem:fk:14-15}).
Weaker versions of this corollary have been applied to cyclotomic numerical semigroups \cite{CGM}.

Our aim at this point is to find conditions that are easy to check and allow us to 
conclude that a given polynomial is not Kronecker. In the following elementary result an example of such a condition
occurs.

\begin{prop}\label{evalkronecker}
Let $f$ be a Kronecker polynomial. Then
\begin{enumerate}
\item $f(1) \ge 0$.
\item If $f(0) \ne 0$ and $f(1) \ne 0$, then $f(-1) \ge 0$. Furthermore, if $f(-1) > 0$, then $f(x) > 0$ for all real $x$.
\end{enumerate}
\end{prop}
\begin{proof}$~$\\
{\rm a)} We have $f(1) \ge 0$ by \eqref{cyclofactor} and Lemma \ref{valueat1A}.\\
{\rm b)} Note that $e_0 = 0$ and $1 \not\in \mathcal{D}$. Recall that $\Phi_n(-1) \ge 0$ for every $n > 1$ (Lemma \ref{valueat-1}). Hence we obtain $f(-1) \ge 0$. Furthermore, if $f(-1) > 0$, then $2 \not\in \mathcal{D}$. Let $x$ be real. We have $\Phi_n(x) > 0$ for every $n > 2$ and, consequently, $f(x) > 0$.\qedhere
\end{proof}

\begin{Example} \label{Coxeter}
One can apply Proposition \ref{evalkronecker} to easily detect self-reciprocal 
poynomials that are not Kronecker.

For every integer $n \ge 3$ let $E_n(x) = 
(x^{n+1}-x^{n-1}-x^{n-2}+x^3+x^2-1)/(x-1)$. These polynomials are known in the literature as Coxeter polynomials (see \cite{Gross}, where the authors 
determine their cyclotomic part). Set $n \ge 6$. One can easily show that
$$E_n(x) = x^n + x^{n-1} - \sum_{k=3}^{n-3}x^k + x + 1.$$
Consequently, $E_n$ is self-reciprocal and $E_n(1) = 9 - n$. Proposition \ref{evalkronecker} implies that $E_n$ is not Kronecker for $n \ge 10$. For $n \le 9$  it turns out that $E_n$ is Kronecker (see \cite[Table 2]{Gross}).
\end{Example}

Of course, there are many non-Kronecker self-reciprocal polynomials that verify Proposition \ref{evalkronecker} and, thus, the argument used in Example \ref{Coxeter} does not always work. This is so 
in case of the polynomials considered in Section \ref{sec:ns}. Therefore, more 
restrictive conditions than the ones given in Proposition \ref{evalkronecker} are needed, which motivated our research towards the results of Section \ref{sec:log-kronecker}.

\subsection{The logarithmic derivatives of Kronecker polynomials at $\pm1$} \label{sec:log-kronecker} 

In this section we apply Theorems \ref{thm:kth-log-deriv:1} and \ref{thm:kth-log-deriv:-1} to obtain several linear equations involving the logarithmic derivatives of Kronecker polynomials at $\pm 1$. Let $f(x)$ be a Kronecker polynomial and let us factorize $f$ as in \eqref{cyclofactor}. Note that 
\begin{equation}
    \label{eq:trivialdegree}
\deg f= e_0 + \sum_{d \in \mathcal{D}} e_d J_1(d) = e_0 + \sum_{d \in \mathcal{D}} e_d \varphi(d). 
\end{equation}
Recall that $\Phi_1$ is anti self-reciprocal but $\Phi_n$ is self-reciprocal for every $n \ge 2$ 
(Lemma \ref{basiceqs} e)). As a consequence we obtain that $f(x)/x^{e_0}$ is self-reciprocal if $e_1$ is even and $f(x)/x^{e_0}$ is anti self-reciprocal otherwise. Thus Proposition \ref{self-reciprocal:deriv} allows us to compute the derivative of $f$ at $\pm1$ in most cases. 
 \begin{cor} \label{cor:log-kronecker:1}
   Let $f$ be a Kronecker polynomial. Let us factorize $f$ as in \eqref{cyclofactor}. 
   \begin{enumerate}
       \item If $f(1) \ne 0$, then $(\log f)'(1) = (\deg f + e_0)/ 2$.
       \item If $f(-1) \ne 0$, then $(\log f)'(-1) = - (\deg f + e_0)/ 2$.
   \end{enumerate}
\end{cor}
Our aim is to generalize the previous result for higher order logarithmic derivatives of $f$.
On logarithmically differentiating both sides of \eqref{cyclofactor}, we obtain
\begin{equation}\label{eq:log-kronecker:kth-deriv}
(\log f)^{(k)}(x) = e_0 \log^{(k)}(x) + \sum_{d \in \mathcal{D}} e_d (\log \Phi_d)^{(k)}(x).
\end{equation}
If $f(1) \ne 0$, then, by Theorem \ref{thm:kth-log-deriv:1} and the fact that $\log^{(k)}(1) = (-1)^{k-1}(k-1)! = s(k,1)$, we have
\begin{equation}\label{eq:log-kronecker:1}
(\log f)^{(k)}(1) - e_0 s(k,1) = \sum_{d \in \mathcal{D}} e_d \sum_{j=1}^{k} \frac{B_j^+s(k,j)}{j} J_j(d)
                                       = \sum_{j=1}^{k} \frac{B_j^+s(k,j)}{j} \sum_{d \in \mathcal{D}} e_d J_j(d).
\end{equation}
For $k=1$ we recover part of Corollary \ref{cor:log-kronecker:1} on recalling \eqref{eq:trivialdegree}. Note that, if $k \ge 2$, then the value $(\log f)^{(k)}(1)$ does not only depend on the degree of $f,$ but also on its factorization \eqref{cyclofactor}. 

\begin{Thm} \label{thm:log-kronecker:k}
  Let $f$ be a Kronecker polynomial with factorization as in \eqref{cyclofactor}. 
  \begin{enumerate}
      \item If $f(1) \ne 0$ then, for each integer $k\ge 2$ we have
      \[ \sum_{j = 1}^k \stirlingtwo{k}{j} (\log f)^{(j)}(1) = \frac{B_{k}^+}{k} \sum_{d \in \mathcal{D}} e_d J_k(d). \]
      \item If $f(-1) \ne 0$, then for each integer $k \ge 2$ we have
      \[ \sum_{j = 1}^k (-1)^{j} \stirlingtwo{k}{j} (\log f)^{(j)}(-1) = \frac{B_{k}^+}{k} \sum_{d \in \mathcal{D}} e_d J_k(d \alpha_d), \]
      where $\alpha_d$ is as in \eqref{alpha}.
  \end{enumerate}
\end{Thm}
\begin{proof}
  Let us consider the vectors $u, v \in \mathbb{Q}^k$ given by 
  \[ u_j = \frac{B_j^+}{j} \sum_{d \in \mathcal{D}} e_d J_j(d) \quad \text{and} \quad v_j = (\log f)^{(j)}(1) - e_0 s(j,1) \] 
  for every integer $j$ with $1 \le j \le k$. Note that \eqref{eq:log-kronecker:1} can be stated as $C u = v$, where $C$ is the $k \times k$ dimensional matrix whose entries are $C_{ij} = s(i,j)$. The inverse of the matrix $C$ turns out to be the $k \times k$ dimensional matrix $S$ whose entries are $S_{ij} = \stirlingtwo{i}{j}$ (see Section \ref{sec:pre:stirling} and, specifically, \eqref{eq:stirling:matrix}). Thus, we have $v = C^{-1} u$ and, in particular,
  \begin{equation*}
    \begin{aligned}
      \frac{B_{k}^+}{k} \sum_{d \in \mathcal{D}} e_d J_k(d) & = \sum_{j = 1}^k \stirlingtwo{k}{j} ((\log f)^{(j)}(-1) - e_0 s(j,1)) \\
                                                                     & = -e_0 \delta_{k,1} + \sum_{j = 1}^k \stirlingtwo{k}{j} (\log f)^{(j)}(-1),
    \end{aligned}
  \end{equation*}
  where we applied \eqref{eq:stirling:matrix:2} for $j = 1$.
  The same argument can be employed to obtain the second equality on observing that if $f(-1) \ne 0$, then 
  \begin{equation}\label{eq:log-kronecker:-1}
  \begin{aligned}
    (-1)^k (\log f)^{(k)}(-1) - e_0 s(k,1) & = \sum_{d \in \mathcal{D}} e_d \sum_{j=1}^{k} \frac{B_j^+s(k,j)}{j} J_j(d \alpha_d) \\ 
                                          & = \sum_{j=1}^{k} \frac{B_j^+s(k,j)}{j} \sum_{d \in \mathcal{D}} e_d J_j(d \alpha_d), 
  \end{aligned}
  \end{equation}
  where we used Theorem \ref{thm:kth-deriv:-1}.
\end{proof}

\begin{cor} \label{cor:log-kronecker:k:cyclo}
  Let $n$ and $k$ be positive integers. 
  \begin{enumerate}
      \item \label{item:log-cyclo} If $n \ne 1$, then
      \[ \sum_{j = 1}^k \stirlingtwo{k}{j} (\log \Phi_n)^{(j)}(1) = \frac{B_{k}^+}{k} J_k(n). \]
      \item If $n \ne 2$, then 
      \[ \sum_{j = 1}^k (-1)^{j} \stirlingtwo{k}{j} (\log \Phi_n)^{(j)}(-1) = \frac{B_{k}^+}{k} J_k(n \alpha_n). \]
  \end{enumerate}
\end{cor}

Recall that $ (\log \Phi_2)^{(j)}(1) = (-1)^{j-1} (j-1)! 2^{-j}$ and $J_k(2) = 2^k -1$. These observations in conjunction with Corollary \ref{cor:log-kronecker:k:cyclo} 
\ref{item:log-cyclo} yield the
identity, due to Worpitzky \cite{Carlitz},
\[ B_{k}^+ =  \frac{k}{2^k -1} \sum_{j = 1}^k (-1)^{j-1} 2^{-j} (j-1)! \stirlingtwo{k}{j}. \]

Recall that $B_k^+ = 0$ for every odd integer with $k \ge 3$ and hence the logarithmic derivatives of Kronecker polynomials at $\pm 1$
verify several homogeneous linear equations. In the cases $k \in \{3, 5\}$ these equations are
\begin{equation}\label{eq:log-kronecker:small}
\left\{
\begin{aligned} 
    -(\deg f)/2 = & 3 (\log f)''(1) + (\log f)'''(1),  \\
    -(\deg f)/2 = & 15 (\log f)''(1) + 25 (\log f)'''(1) + 10 (\log f)^{(4)}(1) + (\log f)^{(5)}(1), \\
    (\deg f)/2 = & 3 (\log f)''(-1) - (\log f)'''(-1),  \\
    (\deg f)/2 = & 15 (\log f)''(-1) - 25 (\log f)'''(-1) + 10 (\log f)^{(4)}(-1) - (\log f)^{(5)}(-1). \\
\end{aligned}
\right.
\end{equation}

In case $B_k^+ \ne 0$, the quantity $\sum_{d \in \mathcal{D}} e_d J_k(d)$ can be 
bounded as follows.
\begin{cor} \label{cor:log-kronecker:bound}
Let $f$ be a Kronecker polynomial with factorization as in \eqref{cyclofactor}. 
\begin{enumerate}
    \item If $f(1) \ne 0$, then for every even integer $k\ge 2$ we have
    \[  (2^{k}-1)(\deg f - e_0) \le \frac{k}{B_k^+} \sum_{j = 1}^k \stirlingtwo{k}{j} (\log f)^{(j)}(1). \]
    \item If $f(\pm 1) \ne 0$, then for every even integer $k\ge 2$ we have
    \[ \frac{(3^{k}-1)}{2} (\deg f - e_0) \le \frac{k}{B_k^+} \sum_{j = 1}^k (-1)^{j} \stirlingtwo{k}{j} (\log f)^{(j)}(-1). \]
\end{enumerate}
\end{cor}
\begin{proof}
Note that $J_k / \varphi$ is a multiplicative function 
satisfying $J_k(d) / \varphi(d)\ge 1$ for every $d\ge 1.$ It follows
that 
$\min \{J_k(d) / \varphi(d) : d \ge 2\} =
\min \{J_k(d) / \varphi(d) : d\ge 2,\Lambda(d)\ne 0\}.$
As $J_k(p^m) / \varphi(p^m) = p^{(m-1)k}(1 + p + \cdots + p^{k-1})$ for every prime power $p^m,$ we see
that the minimum is assumed for $p^m=2,$ leading to
$\min \{J_k(d) / \varphi(d) : d \ge 2\}=2^k-1.$ Since $\min \mathcal D\ge 2$ we obtain
\[ \sum_{d \in \mathcal{D}} e_d J_k(d) \ge  \sum_{d \in \mathcal{D}} e_d \varphi(d) 
\min \left\{\frac{J_k(d)}{\varphi(d)} : d \ge 2\right\}
= (2^{k}-1)(\deg f - e_0). \]
In the second case we have $\min \mathcal D\ge 3$ and we find 
\[ \sum_{d \in \mathcal{D}} e_d J_k(d \alpha_d) \ge  \sum_{d \in \mathcal{D}} e_d \varphi(d \alpha_d) 
\min \left\{\frac{J_k(d \alpha_d)}{\varphi(d \alpha_d)} : d \ge 3\right\}
= \frac{(3^{k}-1)}{2}(\deg f - e_0), \]
where the equality is a consequence of the minimum being assumed for $d=6$ (as shown by a variation of the analysis above) and the equality 
$\varphi(d \alpha_d) = \varphi(d).$

The inequalities in the statement of the corollary now follow on invoking Theorem \ref{thm:log-kronecker:k}. 
\end{proof}
Corollary \ref{cor:log-kronecker:bound} in the cases $k \in \{2, 4\}$ leads to the inequalities
\begin{equation}\label{eq:log-kronecker:small:bound}
\left\{
\begin{aligned} 
     -(\deg f)/4 \le & (\log f)''(1),  \\
     -5(\deg f)/8 \ge & 7 (\log f)''(1) + 6 (\log f)'''(1) + (\log f)^{(4)}(1), \\
     5(\deg f)/6 \le & (\log f)''(-1),  \\
     (\deg f)/6 \ge & 7 (\log f)''(-1) - 6 (\log f)'''(-1) + (\log f)^{(4)}(-1). \\
\end{aligned}
\right.
\end{equation}
If one has more information about the set $\mathcal{D}$, then the lower bounds given at Corollary \ref{cor:log-kronecker:bound} can be improved as the following remark shows. 
\begin{Remark} \label{rem:log-kronecker:bound}
Let $\mathcal C$ be a set containing $1$ such that no $\Phi_c$ with $c$ in $\mathcal C$ divides $f$. Define $\mu_{\mathcal C}(k)=\min\{J_k(j) / \varphi(j) : j\in \mathbb N\backslash \mathcal C\}$.
Then for every even integer $k$ with $k \ge 2$ we have
\[ (2^{k}-1)(\deg f - e_0) \le \mu_{\mathcal C}(k) (\deg f - e_0) \le  \frac{k}{B_k^+} \sum_{j = 1}^k \stirlingtwo{k}{j} (\log f)^{(j)}(1). \]
We have $\mu_{\mathcal C}(1)=1.$ On noting that $J_k(j)/\varphi(j)\ge j$ for $k\ge 2,$ 
$\mu_{\mathcal C}(k)$ can be easily explicitly computed.
\end{Remark}

\subsection{The $k^{th}$ logarithmic derivative of $\Psi_n$ at $-1$}
Put
$$\Psi_n(x)=\frac{x^n-1}{\Phi_n(x)}.$$
The second author named these polynomials 
\emph{inverse cyclotomic polynomials} and 
inititated their study \cite{Inverse}.
As obviously
$\Psi_n(x)=\prod_{d\mid n,\,d<n}\Phi_d(x),$
the inverse cyclotomic polynomials 
are Kronecker polynomials. 
Inverse cyclotomic polynomials behave in many aspects similarly, but also in many aspects dissimilarly from the cyclotomic polynomials. Meanwhile an application in elliptic curve cryptography has been found \cite{crypto}. 

In this section we translate the previous results to the setting of inverse cyclotomic polynomials. First of all, observe that \eqref{eq:nicol2} can be rewritten as
\begin{equation*}
  \sum_{k = 1}^n r_k(n) x^{k-1} = \Psi_n(x) \Phi_n'(x)  \qquad (n \ge 1).
\end{equation*}

Since $\Psi_n(0) = -1$ and $\Psi_n$ is continuous, one can consider the logarithm of $-\Psi_n$ around zero. The following result follows from Nicol's formula \eqref{eq:nicol} and the Taylor series of $\log(1-x^n)$ around $x=0$.

\begin{cor}
  Let $n\ge 2$ be an integer. Then 
  \[ \log(-\Psi_n)(x) = \sum_{j = 1}^\infty \frac{r_j(n)}{j} x^j - \sum_{j = 1}^\infty \frac{x^{nj}}{j}  \qquad (|x| < 1). \]
  In particular, for each positive integer $k$ we have
  \[ (\log \Psi_n)^{(k)}(0) = \begin{cases}
    (k-1)! r_k(n) & \text{if } n \nmid k; \\
    (k-1)! (r_k(n)-n) & \text{if } n \mid k.
  \end{cases} \]
\end{cor}

Note that $\Psi_n(-1)=0$ if and only if $n$ is even,
and so the best we can hope for is
the determination of the logarithmic derivatives of $\Psi_n$ at $-1$ for 
all odd integers $n.$

\begin{Thm} \label{thm:kth-log-derivinverse:1}
  Let $k \ge 1$ and $n\ge 3$ be integers with $n$ odd. Then
  \[(\log \Psi_n)^{(k)}(-1) = (-1)^k \sum_{j = 1}^k \frac{B_j^+ s(k,j)(2^j-1)}{j}(n^j- J_j(n)). \]
\end{Thm}
\begin{proof}
  The result follows on invoking \eqref{eq:log-kronecker:-1} and noting that for every positive integer $j$ we have
  \[ \sum_{d \mid n,\, d \ne n} J_j(2 d) =  J_j(2)\sum_{d \mid n,\, d \ne n} J_j(d) =J_j(2)\Big(\sum_{d \mid n} J_j(d) - J_j(n)\Big) = (2^j-1)(n^j - J_j(n)). \qedhere \]
 \end{proof}
The following corollary is a consequence of the previous identity and its proof is 
similar to that of Theorem \ref{thm:log-kronecker:k}.
\begin{cor}
  Let $k \ge 1$ and $n\ge 3$ be integers with $n$ odd. Then
  \[ \sum_{j = 1}^k (-1)^{j} \stirlingtwo{k}{j} (\log \Psi_n)^{(j)}(-1) = \frac{B_{k}^+(2^k-1)}{k} (n^k - J_k(n)). \]
\end{cor}

\section{Symmetric non-cyclotomic numerical semigroups} \label{sec:ns}

In this section we find a family of symmetric numerical semigroups that are not cyclotomic, proving Theorem \ref{thm:ciolan} as a result. We use the standard terminology of the theory of numerical semigroups, see Section \ref{sec:pre:ns}.

Let $k$ be a positive integer. Let us consider the numerical semigroup 
$$S_k =\{0,k,k+1,k+2,\dots\} \setminus \{2k-1\}.$$
From the definition of $S_k$ it follows that $\mathrm{g}(S_k) = k$ and $\mathrm{F}(S_k) = 2k-1$. It is not difficult to show that the minimal system of generators of $S_k$ is $\{k, k+1, \ldots, 2k-2\}$ and, in particular, $\mathrm{e}(S) = k-1$. Moreover, $S_k$ has the following semigroup polynomial
\begin{equation}
f_k(x) = \mathrm{P}_{S_k}(x) = 1 - x + x^k - x^{2k-1} + x^{2k}
\end{equation}
and hence is symmetric (see Section \ref{sec:pre:ns}). Indeed, the semigroup $S_k$ is the root of the tree of symmetric numerical semigroups with Frobenius number $2k-1$, see \cite{tree-irred}.
In order to establish Theorem \ref{thm:ciolan}, 
in view of the properties of 
$S_k$, it suffices to prove the following result.

\begin{Thm} \label{thm:pedro-family}
  The symmetric numerical semigroup $S_k$ is not cyclotomic for every $k \ge 5$.
\end{Thm}

Note that $f_k(1) = 1$ and $f_k(-1) = 4+(-1)^k$. Thus, the arguments given in Example \ref{Coxeter} can not be reproduced in this case. Indeed, it can be shown that the polynomials $f_k$ satisfy the equations \eqref{eq:log-kronecker:small}. Our proof of Theorem \ref{thm:pedro-family} uses a refined version of the lower bound given in Remark \ref{rem:log-kronecker:bound}.
First, we need the following lemmas. For notational convenience we use $k\equiv a,b\pmod*{m}$ to mean that either $k\equiv a\pmod*{m}$ or $k\equiv b\pmod*{m}$.

\begin{lem} \label{lem:fk:single-roots}
Put $ \zeta_n=e^{2\pi i/n}. $ 
We have 
$f_k(\zeta_{6})=0$ if and only if $k\equiv 1,3\pmod*{6},$
$f_k(\zeta_{10})=0$ if and only if $k\equiv 2,4\pmod*{10}$ and
$f_k(\zeta_{12})=0$ if and only if $k\equiv 3,4\pmod*{12}.$
Furthermore, we have $ f_k'(\zeta_6)\ne 0, $ $ f_k'(\zeta_{10})\ne 0 $ and $ f_k'(\zeta_{12})\ne 0. $
\end{lem}
\begin{proof}
This is a tedious verification. To show, e.g., that $ f_k'(\zeta_6)\ne 0, $ we distinguish 6 cases, depending on the residue of $ k  $ modulo 6. For example, if $ k\equiv 1\pmod 6 ,$ then 
\begin{equation*}
f_k'(\zeta_6)=-1+k-(2k-1)+2k\zeta_6 ^5 
=-1+k-(2k-1)+2k(1-\zeta_6) =k(1-2\zeta_6)\ne 0.\qedhere
\end{equation*}
\end{proof}
\begin{cor}
\label{cor:60cases}
For $k\ge 1$ we have
$$
(f_k,\Phi_6 \Phi_{10}\Phi_{12})
=
\begin{cases}
 \Phi_6 & \text{ if }k\equiv 1,7,9\pmod*{12};\\
 \Phi_{10} & \text{ if }k\equiv 2,4\pmod*{10},\,k\not\equiv 4\pmod*{12};\\
 \Phi_6\Phi_{12} & \text{ if }k\equiv 3\pmod*{12};\\
  \Phi_{12} & \text{ if }k\equiv 4\pmod*{12},\,k\not\equiv 4,52\pmod*{60};\\
   \Phi_{10}\Phi_{12} & \text{ if }k\equiv 4,52\pmod*{60};\\
1 & \text{ otherwise.}
\end{cases}
$$
Furthermore, $\Phi_6^2\nmid f_k,$ 
$\Phi_{10}^2\nmid f_k$ and $\Phi_{12}^2\nmid f_k.$
\end{cor} 

\begin{lem} \label{lem:fk:14-15}
Let $k$ and $j$ be positive integers.
\begin{enumerate}
\item The cyclotomic polynomial $\Phi_{p^j}$ does not divide $f_k$ for every prime $p.$
\item The cyclotomic polynomial $\Phi_{2p^j}$ does not divide $f_k$ for every prime $p$ with $p \ne 3, 5$.
\item The cyclotomic polynomial $\Phi_{3p^j}$ does not divide $f_k$ for every odd prime $p.$ 
\end{enumerate}
\end{lem}
\begin{proof}~
\begin{enumerate}
\item We have $f_k(1) = 1$ and hence the conclusion follows by Corollary \ref{reformulation}.
\item We have $f_k(-1) = 4 + (-1)^k$. Consequently, it follows on invoking Corollary \ref{reformulation}
that $\Phi_{2p^j}$ does not divide $f_k$ for every prime $p$ with $p \ne 3, 5.$ \\
\item We have
\begin{equation*}
f_k(\zeta_3)=\begin{cases}4 & \hbox{if } k \equiv 0 \pmod*{3};\\ -2 \zeta_3 & \hbox{if } k \equiv 
1\pmod*{3};\\ \zeta_3^{-1} & \hbox{if } k \equiv 2 \pmod*{3}.\end{cases}
\end{equation*}
It follows that there is no odd prime number dividing $|f_k(\zeta_3)|$. Consequently, Corollary \ref{reformulation} implies that there is no odd prime $p$ such that $\Phi_{3p^j}$ divides $f_k.$ \qedhere
\end{enumerate}
\end{proof}

Now we can prove Theorem \ref{thm:pedro-family}. We make use of the following table. Recall that $J_2 / \varphi = \Psi$.

\begin{table}[H]
\centering
\begin{tabular}{|c|c|c||c|c|c|}
\hline
$n$ & $\Psi(n)$ & $\varphi(n)$ & $n$ & $\Psi(n)$ & $\varphi(n)$\\
\hline
6  & 12 & 2 & 21 & 32 & 12 \\
\hline
10 & 18 & 4 & 22 & 36 & 10 \\
\hline
12 & 24 & 4 & 24 & 48 & 8  \\
\hline
14 & 24 & 6 & 26 & 42 & 12 \\
\hline
15 & 24 & 8 & 28 & 48 & 12 \\
\hline
18 & 36 & 6 & 30 & 72 & 8  \\
\hline
20 & 36 & 8 & 33 & 48 & 20  \\
\hline
\end{tabular}
\caption{First values of $\Psi(n)$ and $\varphi(n)$ when $n$ is not a prime power.}
\label{table:arithmetic}
\end{table}

\begin{proof}[Proof of Theorem \ref{thm:pedro-family}]
Our proof is by contradiction. So we
assume that $S_k$ is cyclotomic
and hence we can write $f_k(x) = \prod_{d \in \mathcal{D}_k} \Phi_d(x)^{e_d}$.
Put $F_k=\sum_{d \in \mathcal{D}_k} e_d \Psi(d) \varphi(d)$.
By Theorem \ref{thm:log-kronecker:k} we have
\[ F_k = \frac{2}{B_2^+}((\log f_k)'(1) + (\log f_k)''(1)). \]\\
On noting that $f_k''(1) = k^2 + 3k - 2$ and $(\log f_k)'' (1)= 3k - 2,$ 
we compute that 
\[ F_k = 48k - 24. \]
Set 
$\mathcal C = \{p^k: p \text{ is a prime}\} \cup \{2p^k: p \text{ is a prime and } p \ne 3, 5\}.$ 
Note that  $\mu_{\mathcal C}(2) = 12$ (cf. Table \ref{table:arithmetic}) and 
that $\mathcal{D}_k \cap \mathcal C = \emptyset$ by Lemma \ref{lem:fk:14-15}.
We distinguish three cases:
\begin{enumerate}
\item $\Phi_6$ and $\Phi_{10}$ do not divide $f_k$. Set ${\mathcal C}_1 = \mathcal C \cup \{6, 10\}$. We have $\mu_{{\mathcal C}_1}(2) = 24$ 
(cf. Table \ref{table:arithmetic}) and hence $24 \deg f_k = 48 k \le F_k$ by Remark \ref{rem:log-kronecker:bound}, a contradiction.
\item $\Phi_6$ divides $f_k$ and $\Phi_{10}$ does not divide $f_k$. Corollary \ref{cor:60cases} tells us that $\Phi_6$ divides $f_k$ exactly. There are two possibilities:
\begin{enumerate}[label=\roman*)]
\item There exists $d_0 \ge 18$ such that $\Phi_{d_0} \mid f_k$. In this case we can find a better lower bound for $F_k$ than the one given by Remark \ref{rem:log-kronecker:bound} for $\mu_{\mathcal C}(2) = 12$. First, we apply the inequality $\Psi(d) \ge 24$ for every $d \in \mathcal{D}_k \setminus \{6, d_0\}$ to Theorem \ref{thm:log-kronecker:k}. This leads to 
$$ 
 F_k = \sum_{d \in \mathcal{D}_k} e_d \varphi(d)\Psi(d) \ge 24 \deg f_k + e_{d_0}\varphi(d_0)(\Psi(d_0) - 24) + \varphi(6) (\Psi(6) - 24),
$$
which on noting that $\Psi(d_0) \ge 32$ and $\varphi(d_0) \ge 6$ (see Table \ref{table:arithmetic}) yields
\begin{equation*}
F_k \ge 24 \deg f_k + 72 e_{d_0} - 24 \ge 48k + 48,
\end{equation*}
contradicting the fact that $F_k = 48k -24$.

\item $d \le 15$ for every $d \in \mathcal{D}_k$. In light of Lemma \ref{lem:fk:14-15} we have $\mathcal{D}_k \subseteq \{6, 12\}$. Furthermore,  Corollary \ref{cor:60cases} tells us that $0 \le e_6,e_{12} \le 1$. Consequently, we only have two possibilities, $f_k = \Phi_6$ and $f_k = \Phi_6 \Phi_{12}$, which are only reached for $k = 1$, respectively $k = 3$. 
\end{enumerate}
\item $\Phi_6$ does not divide $f_k$ and $\Phi_{10}$ divides $f_k$. The proof is similar to the one given in the previous case. We obtain that the only possibilities are $f_k = \Phi_{10}$ and $f_k = \Phi_{10} \Phi_{12}$, which are only reached for $k = 2$, respectively $k = 4$. 
\end{enumerate}
Since by Corollary \ref{cor:60cases} the case where both $\Phi_6$ and $\Phi_{10}$ divide $f_k$ does not occur, the proof is completed.
\end{proof}

The numerical semigroup $S_k$ is cyclotomic for $k \le 4$ (see Table \ref{table:factorization}).

\begin{table}[H]
\centering
\begin{tabular}{|c|c||c|c||c|c|}
\hline
$k$ & Factorization of $f_k$ & $k$ & Factorization of $f_k$ & $k$ & Factorization of $f_k$\\
\hline
1 & $\Phi_6$ &  7 & $\Phi_6 (f_7 / \Phi_6)$ &
13 & $\Phi_6 (f_{13} / \Phi_6)$ \\
\hline
2 & $\Phi_{10}$ & 8 & $f_8$ & 14 & $\Phi_{10} (f_{14} / \Phi_{10})$ \\
\hline
3 & $\Phi_6 \Phi_{12}$ & 9 & $\Phi_6 (f_9 / \Phi_6)$ & 15 & $\Phi_6 \Phi_{12} (f_{15} / (\Phi_{6} \Phi_{12}))$  \\
\hline
4 & $\Phi_{10} \Phi_{12}$ & 10 & $f_{10}$ &  16 & $\Phi_{12} (f_{16} / \Phi_{12})$ \\
\hline
5 & $f_5$ & 11 & $f_{11}$ & 17  & $f_{17}$ \\
\hline
6 & $f_6$ & 12 & $\Phi_{10} (f_{12} / \Phi_{10})$ & 18 & $f_{18}$ \\
\hline
\end{tabular}
\caption{Factorization of $f_k$ into irreducibles over the rationals.}
\label{table:factorization}
\end{table}

\subsection{On the irreducible factors of  $1 - x + x^k - x^{2k-1} + x^{2k}$}
\label{sec:factoring}

Our work allows us to deduce
that $f_k$ is not Kronecker for 
every $k\ge 5,$ but says nothing else about
the factorization of the polynomial
$f_k/(f_k,\Phi_6 \Phi_{10}\Phi_{12}).$
By using a computer algebra package 
we factorized $f_k$ for $k\le 1000$.
\begin{compfact}
For $5\le k\le 1000$ the polynomial 
$f_k/(f_k,\Phi_6 \Phi_{10}\Phi_{12})$ is irreducible of
degree at least $12$ and, moreover, non-cyclotomic.
\end{compfact}

We are tempted to make the following conjecture.
\begin{conj}
For $k\ge 1$ the polynomial $f_k/(f_k,\Phi_6\Phi_{10}\Phi_{12})$ is irreducible over 
the rationals and, moreover, non-cyclotomic. 
\end{conj}

In this regard, Curtis T. McMullen pointed out to us the following reasoning. Consider the Laurent polynomial $f(x,y) = 1 - x + y - y^2/x + y^2$. Note that $f_k(x) = f(x, x^k)$. The polynomial $x f(x,y) = x - x^2 + xy - y^2$ is irreducible in $\mathbb{Q}[x][y]$ due to Eisenstein's criterion, and, thus, so is the Laurent polynomial $f(x,y)$. Therefore, a result of Laurent implies that there are only finitely many pairs of roots of unity $(\xi_n, \xi_m)$ such that $f(\xi_n, \xi_m) = 0$ \cite[Th\'eor\`eme 1]{Laurent}. It follows that there is a positive integer $N$ such that $\gcd(f_k, \Phi_n) = 1$ for every $n \ge N$ and $k \ge 1$.

A stronger conclusion can be obtained by an elementary approach suggested to us by Michael Filaseta. Note that
\begin{align*}
\zeta_n^{-k} f_k(\zeta_n) & = \zeta_n^k + \zeta_n^{-k} - \zeta_n^{k-1} - \zeta_n^{-k+1} + 1 \\
                          & = 2 \cos(2 \pi k / n) - 2 \cos(2 \pi (k-1) / n) + 1.
\end{align*}
Therefore, $f_k(\zeta_n) = 0$ if and only if $1/2 = \cos(2 \pi (k-1) / n) - \cos(2 \pi k / n)$. For every $n \ge 4$ we have
\begin{align*}
   \left|\cos(2 \pi (k-1) / n) - \cos(2 \pi k / n)\right| & = \left|(\cos(2 \pi / n) - 1) \cos(2 \pi k / n) - \sin(2 \pi / n) \sin(2 \pi k / n) \right| \\
   & \le (1 - \cos(2 \pi / n)) \left|\cos(2 \pi k / n)\right| + \sin(2\pi / n) \left|\sin(2 \pi k / n)\right| \\
   & \le 1 - \cos(2 \pi / n) + \sin(2\pi / n).
\end{align*}
The function $1-\cos(x)+\sin(x)$ is increasing in $[0,\pi/2]$. One can check that $1 - \cos(2 \pi / 15) + \sin(2\pi / 15) = 0.49319\ldots < 1/2$. Therefore, for every $n \ge 15$, we conclude that $f_k(\zeta_n) \ne 0$ for every $n \ge 15$. This assertion in conjunction with Lemma \ref{lem:fk:14-15} yields that the only cyclotomic polynomials that may divide $f_k$ are $\Phi_6$, $\Phi_{10}$ and $\Phi_{12}$. Recall that Corollary \ref{cor:60cases} studies these possible factors. Thus we have proved the following proposition.

\begin{prop}
For $k\ge 1$ the polynomial $f_k/(f_k,\Phi_6\Phi_{10}\Phi_{12})$ does not have cyclotomic factors. 
\end{prop}

The polynomials $f_k$ are of the form $x^a - x^b + x^c - x^d + 1$. These polynomials have been studied in \cite{Filaseta}, where the authors conclude that their non-reciprocal part is either irreducible or $1$ \cite[Corollary 1.3]{Filaseta}. 
Recall that 
the non-reciprocal part of $f$ is 
defined as
$f$ divided by its self-reciprocal irreducible factors. Since $f_k$ is self-reciprocal, its non-reciprocal part is obviously equal 
to $1$.

\section{Appendix by Pedro A. Garc\'ia-S\'anchez: alternative proof of Theorem \ref{thm:ciolan} \ref{item:ciolan:f}}
A quick search with the \texttt{gap} \cite{gap} package \texttt{numericalsgps} \cite{numericalsgps}, shows 
that the numerical semigroup $\langle 5,6,7,8\rangle$ is the symmetric numerical semigroup 
having smallest possible Frobenius number (and thus genus and conductor) that is not cyclotomic. Its Frobenius number is nine. Thus it is natural question to ask whether or not  for every odd integer greater than 
seven there exists a symmetric numerical semigroup that is not cyclotomic with this number as its Frobenius number. The answer is yes, and the proof we give here is based 
on the construction of the tree 
of irreducible numerical semigroups with given genus (and thus given Frobenius number) presented in \cite{tree-irred}.

Let $k$ be a positive integer with $k \ge 3$. The numerical semigroup 
\[ S_k = \{0,k,k+1,k+2,\dots\} \setminus \{2k-1\} = \left<k,k+1,\ldots,2k-2\right> \]
turns out to be symmetric and its Frobenius number is $F = 2k-1$. Moreover, we have $\mathrm P_{S_k}(x)=1-x+x^{k}-x^{2k-1}+x^{2k}$ (see Section \ref{sec:ns}). 

As $F/2<k<F$, $2k-F = 1 \not \in S$, $3k \neq 2F$, $4k \neq 3F$ and $F-k < k = \mathrm m(S)$, we know by \cite[Theorem 2.9]{tree-irred} that $S'=(S\setminus \{k\})\cup\{F-k\}$ is again symmetric.
Here we get $S'=\{0,k-1,k+1,\ldots,2k-2,2k,\ldots\}$
and 
\[ \mathrm P_{S'}(x)=1-x+x^{k-1}-x^k+x^{k+1}-x^{2k-1}+x^{2k}. \]

Let us now construct a son of $S'$ in the tree of irreducible numerical semigroups with Frobenius number $F$. Assume that $F>15$ (and so $k>8$). Then $k+2$ 
is a minimal generator of $S'$ such that $F/2<k+2<F$, $2(k+2)-F=5\not\in S$, $3(k+2)\neq 2F$, $4(k+2)\neq 3F$ and $F-(k+2)<k-1=\mathrm m(S')$. Hence we can apply again \cite[Theorem 2.9]{tree-irred} to obtain that $S''=(S' \setminus \{k+2\})\cup\{F-(k+2)=k-3\}$ is a symmetric numerical semigroup with Frobenius number $F$. 
We have
$\{0,k-3,k-1,k+1,k+3,\ldots, 2k-2,2k\}$, 
and so 
\[\mathrm P_{S''}(x)=1-x+x^{k-3}-x^{k-2}+x^{k-1}-x^k+x^{k+1}-x^{k+2}+x^{k+3}-x^{2k-1}+x^{2k}.\]

Assume now that $k$ is even (and thus $F \equiv 3({\rm mod~}4)$). Then $\mathrm P_{S''}(-1)=1+1-1-1-1-1-1-1-1+1+1=-3<0$, and in light of Proposition \ref{evalkronecker}, $S''$ is not cyclotomic. 
This proves the following result.

\begin{prop}
Let $F>15$ be an integer with $F\equiv 3({\rm mod~}4)$. Then 
there exists a symmetric numerical semigroup with 
Frobenius number $F$ that is not cyclotomic.
\end{prop}

Now assume that $k=(F+1)/2$ is odd (and thus $F\equiv 1({\rm mod~}4)$) and $F>9$. Then $k+1$ is a minimal generator of $S_k$ with $F/2<k+1<F$, $2(k+1)-F=3\not\in S_k$, $3k\neq 2F$, $4k\neq 3F$, and $F-(k+1)<\mathrm m(S_k)$. Hence $\bar{S}=(S_k \setminus\{k+1\})\cup\{F-(k+1)=k-2\}$ is a symmetric numerical semigroup with Frobenius number $F$. We have $\bar{S} = \{0,k-2,k,k+2,\ldots,2k-2, 2k, \ldots\}$.  
Thus 
\[
\mathrm P_{\bar{S}}(x)=1-x+x^{k-2}-x^{k-1}+x^k-x^{k+1}+x^{k+2}-x^{2k-1}+x^{2k},
\]
and $\mathrm P_{\bar{S}}(-1)=1+1-1-1-1-1-1+1+1=-1$.

\begin{prop}
Let $F>9$ be an integer satisfying $F\equiv 1({\rm mod~}4)$. Then there exists a symmetric numerical semigroup with Frobenius number $F$ that is not cyclotomic.
\end{prop}

By using \texttt{numericalsgps} we see that for Frobenius number $11$ and $15$ there is a symmetric numerical semigroup that
is not cyclotomic.

{\small
\begin{verbatim}
gap> MinimalGenerators(First(IrreducibleNumericalSemigroupsWithFrobeniusNumber(11), 
    s->not(IsCyclotomicNumericalSemigroup(s))));
[ 5, 7, 8, 9 ]
gap> MinimalGenerators(First(IrreducibleNumericalSemigroupsWithFrobeniusNumber(15), 
    s->not(IsCyclotomicNumericalSemigroup(s))));
[ 6, 7, 10, 11 ]
\end{verbatim}
}

Combining this information with the  two propositions above, one obtains an alternative proof of Theorem \ref{thm:ciolan} \ref{item:ciolan:f}.

\bigskip
\noindent {\tt Acknowledgements}. 
Part of this paper was developed during a one month internship in 2016 of the first author at the Max Planck Institute for Mathematics in Bonn and a subsequent visit in 2017. He would like to thank the second author for the opportunity given and the staff for their hospitality. Furthermore, he is grateful to the second author and Pedro A. Garc\'ia-S\'anchez for their teachings and guidance.

The first author is supported by the fellowship ``Beca de Iniciaci\'on a la Investigaci\'on para Estudiantes de Grado del Plan Propio 2017'' (Vicerrectorado de Investigaci\'on y Transferencia, Universidad de Granada).

The present paper has \cite{BHM} as one of its pilars.
We thank our coauthor Bartolomiej Bzd\c{e}ga for his substantial
contribution to that project.
The authors would 
further like to thank Alexandru Ciolan for proofreading
an earlier version. Last, but not least, the authors thank 
Michael Filaseta and Curtis T.
McMullen for extensive feedback regarding the factorization of 
the fewnomial $1 - x + x^k - x^{2k-1} + x^{2k}.$ This feedback
is now part of Section \ref{sec:factoring}. Andrej Schinzel kindly suggested
us to contact Michael Filaseta.

\end{document}